\documentclass[12pt]{amsart}
\usepackage[shortlabels]{enumitem}
\usepackage{amssymb}
\usepackage{mathrsfs}
\usepackage{stmaryrd}
\usepackage[all]{xy}
\usepackage{xcolor}
\usepackage[margin=1in]{geometry} 
\usepackage[bookmarks, bookmarksdepth=2, colorlinks=true, linkcolor=blue, citecolor=blue, urlcolor=blue]{hyperref}
\usepackage{eucal}

\numberwithin{equation}{section}
\newtheorem{theorem}[equation]{Theorem}

\newtheorem{proposition}[equation]{Proposition}
\newtheorem{lemma}[equation]{Lemma}
\newtheorem{corollary}[equation]{Corollary}

\theoremstyle{definition}
\newtheorem{rmk}[equation]{Remark}
\newenvironment{remark}[1][]{\begin{rmk}[#1] \pushQED{\qed}}{\popQED \end{rmk}}
\newtheorem{eg}[equation]{Example}
\newenvironment{example}[1][]{\begin{eg}[#1] \pushQED{\qed}}{\popQED \end{eg}}
\newtheorem{defnaux}[equation]{Definition}
\newenvironment{definition}[1][]{\begin{defnaux}[#1]\pushQED{\qed}}{\popQED \end{defnaux}}

\newcommand{\bA}{\mathbf{A}}

\newcommand{\bF}{\mathbf{F}}

\newcommand{\bQ}{\mathbf{Q}}

\newcommand{\bT}{\mathbf{T}}

\newcommand{\bV}{\mathbf{V}}

\newcommand{\fm}{\mathfrak{m}}

\newcommand{\arxiv}[1]{\href{http://arxiv.org/abs/#1}{{\tiny\tt arXiv:#1}}}
\newcommand{\DOI}[1]{\href{http://doi.org/#1}{\color{purple}{\tiny\tt DOI:#1}}}
\newcommand{\stacks}[1]{\cite[\href{http://stacks.math.columbia.edu/tag/#1}{Tag~#1}]{stacks}}
\newcommand{\defn}[1]{\emph{#1}}
\let\ol\overline

\renewcommand{\phi}{\varphi}

\DeclareMathOperator{\im}{im}

\DeclareMathOperator{\Mod}{Mod}
\DeclareMathOperator{\Free}{Free}
\DeclareMathOperator{\Hom}{Hom}

\DeclareMathOperator{\gr}{gr}
\DeclareMathOperator{\init}{init}
\DeclareMathOperator{\Spec}{Spec}

\newcommand{\id}{\mathrm{id}}

\renewcommand{\Vec}{\mathrm{Vec}}
\newcommand{\GL}{\mathbf{GL}}
\DeclareMathOperator{\Sym}{Sym}
\DeclareMathOperator{\Frac}{Frac}
\DeclareMathOperator{\chr}{char}
\DeclareMathOperator{\str}{str}
\DeclareMathOperator{\red}{red}
\DeclareMathOperator{\rad}{rad}
\DeclareMathOperator{\astr}{astr}
\newcommand{\Sh}{\mathrm{Sh}}
\newcommand{\diag}{\mathrm{diag}}
\newcommand{\comment}[1]{}

\title{The geometry of polynomial representations in positive characteristic}

\author{Arthur Bik}
\address{University of Neuch\^atel, Switzerland, and Institute for Advanced Study, Princeton, NJ, USA}
\email{\href{mailto:mabik@ias.edu}{mabik@ias.edu}}
\urladdr{\url{http://arthurbik.nl}}

\author{Jan Draisma}
\address{University of Bern, Switzerland, and Eindhoven University of Technology, The Netherlands}
\email{\href{mailto:jan.draisma@math.unibe.ch}{jan.draisma@math.unibe.ch}}
\urladdr{\url{https://mathsites.unibe.ch/jdraisma/}}

\author{Andrew Snowden}
\address{Department of Mathematics, University of Michigan, Ann Arbor, MI, USA}
\email{\href{mailto:asnowden@umich.edu}{asnowden@umich.edu}}
\urladdr{\url{http://www-personal.umich.edu/~asnowden/}}

\thanks{AB was partially supported by Postdoc.Mobility Fellowship number P400P2\_199196 from the Swiss National Science
Foundation and a grant from the Simons Foundation (816048, LC).
JD was partially supported by the Vici grant 639.033.514 from the Netherlands Organisation for Scientific Research and by Swiss National Science Foundation project grant 200021\_191981.
AS was supported by NSF grant DMS-2301871.}

\date{June 11, 2024}

\begin{document}

\begin{abstract}
A \defn{$\GL$-variety} is a (typically infinite dimensional) variety modeled on the polynomial representation theory of the general linear group. In previous work, we studied these varieties in characteristic~0. In this paper, we obtain results in positive characteristic: for example, we prove a version of Chevalley's theorem on constructible sets. We give an application of our theory to strength of polynomials.
\end{abstract}

\maketitle
\tableofcontents

\section{Introduction}

A \defn{$\GL$-variety} is a (typically infinite dimensional) algebraic variety modeled on the polynomial representation theory of the infinite general linear group $\GL$. These objects have surprising finiteness properties \cite{draisma,BDD}, which have been applied to Stillman's conjecture \cite{DLL,ESS1} and the geometry of tensors \cite{imgclosure,unirat}. In \cite{polygeom}, we established fundamental structural results for $\GL$-varieties in characteristic~0. In this paper, we extend these results to positive characteristic.

\subsection{An application}

Before stating our main results about $\GL$-varieties, we give a concrete application to strength; this application was the main source of motivation for this paper.

Fix a field $K$. Let $f \in K[x_1, \ldots, x_n]$ be a homogeneous form of degree $d>1$. The \defn{strength} of $f$, denoted $\str(f)$, is the minimal $s$ for which there exists an expression
\begin{displaymath}
f = \sum_{i=1}^s g_i \cdot h_i
\end{displaymath}
where $g_i$ and $h_i$ are homogeneous forms over $K$ of degrees $<d$.
This notion was first seriously considered by Schmidt \cite{Schmidt}
(and is therefore sometimes called \defn{Schmidt rank}), who used it in
the context of analytic number theory. Ananyan and Hochster \cite{AH}
showed that polynomials of high strength have various nice algebraic
properties, and this formed the basis of their proof of Stillman's
conjecture (they also introduced the term ``strength'').  Strength also
plays an important role in algebraic combinatorics.  Indeed, over finite
fields $K$ of characteristic $>d$, the strength of a degree-$d$ form
$f$ is closely related to its \defn{analytic rank}, which measures the
statistical bias of $f$ regarded as a function $K^n \to K$.  This line of
research goes back to Green and Tao \cite{GT}, a polynomial upper bound
for strength in terms of analytic rank was found by Mili{\'c}evi{\'c}
\cite{M}, and this was further improved by Cohen, Moshkovitz, and Zhu to
(almost) linear bounds \cite{CM,MZ}. The connection between strength
and bias has been exploited to establish further algebraic properties
of high-strength forms by Kazhdan, Lampert, Polishchuk, and Ziegler
\cite{KaZ,KLP,LZ1,LZ2}. Interestingly, this is an entirely different
route to such algebraic properties than the one we take here. 

The strength of a polynomial can drop over an extension field. For example, $f=x_1^2+x_2^2$ has strength~2 over the real numbers, but strength~1 over the complex numbers due to the factorization
\begin{displaymath}
f = (x_1+ix_2)(x_1-i x_2), \qquad i = \sqrt{-1}.
\end{displaymath}
We define the \defn{absolute strength} of $f$, denoted $\astr(f)$, to be the strength of $f$ over the algebraic closure of $K$. We thus have $\astr(f) \le \str(f)$.

Certain results, especially those of a geometric nature, require large
absolute strength. For example, \cite[Theorem~F]{AH} shows that if $f$
has large absolute strength then the singular locus of $f=0$ has large
codimension. However, absolute strength can be hard to get a handle
on algebraically. It is therefore desirable to have a lower bound for
absolute strength in terms of strength; i.e., we would like to know
that strength cannot drop too much by passing to an extension. Such a
result has been established when the characteristic is zero or large (see
\cite{BDLZ}; special cases were established earlier in \cite{LZ2,KP}):

\begin{theorem} \label{thm:astr}
Suppose that $\chr(K)=0$ or $\chr(K)>d$. Given any $s$ there is some $s'$
(depending only on $d$ and $s$, not on $K$) such that if $f \in K[x_1,
\ldots, x_n]$ is a homogeneous form of degree $d$ with $\str(f)>s'$
then $\astr(f)>s$. Furthermore, for fixed $d$, the minimal such $s'$
 is bounded by a polynomial function of $s$.
\end{theorem}

We note in particular that the $s$ in the above theorem is independent of the number of variables $n$; this is really the main point. The conclusion in the theorem is false (in general) when $\chr(K)=p>0$ and $d \ge p$, as the following example shows:

\begin{example} \label{ex:nonsemiperfect}
Let $K=\bF_p(t_1,t_2,\ldots)$ and let $f=t_1x_1^p+\cdots+t_nx_n^p$. Then $\str(f) \ge n/2$, but $\astr(f)=1$. See \cite[Example~3.4]{BDS1} for details.
\end{example}

Recall that a field $K$ is \defn{perfect} if $\chr(K)=0$, or
$\chr(K)=p$ and $K=K^p$. We say that a field $K$ is
\defn{semi-perfect} if $\chr(K)=0$, or $\chr(K)=p$ and
$[K:K^p]<\infty$. The function field of a (finite dimensional)
algebraic variety over a finite field is semi-perfect, while the field
$K$ in the above example is not semi-perfect. An important application
to strength of our more general results below shows that the phenomenon in Example~\ref{ex:nonsemiperfect} is the only obstruction to extending Theorem~\ref{thm:astr}:

\begin{theorem} \label{thm:astr2}
Suppose $K$ is semi-perfect and $d>1$ is arbitrary. Given any $s$ there is some~$s'$ (depending only on $K$, $d$, and $s$) such that if $f \in K[x_1, \ldots, x_n]$ is a homogeneous form of degree $d$ with $\str(f)>s'$ then $\astr(f)>s$.
\end{theorem}

Example~\ref{ex:nonsemiperfect} shows that one cannot bound $s'$ in
terms of $s$ and $d$ independently of $K$. It is conceivable, however,
that $s'$ can be bounded by a polynomial in $s$ when both $K$ and $d$
are fixed. We do not currently have such a quantitative result.

In fact, we prove a stronger result than Theorem~\ref{thm:astr2}: we show that if $\astr(f)=s$ then there is a finite extension $K'/K$, whose degree can be controlled by $s$ and $d$, such that $f$ has strength $s$ over $K'$. This result is new even in characteristic~0.

Theorem~\ref{thm:astr2} plays a critical role in \cite{BDS1} when the field has positive characteristic.

\subsection{Main results}

Fix an infinite field $K$; in the body of the paper we also allow
finite fields, but this requires some additional care. Let
$\GL=\bigcup_{n \ge 1} \GL_n(K)$ be the infinite general linear group
and let $\bV=\bigcup_{n \ge 1} K^n=\langle e_1, e_2,\ldots \rangle_K$ be its standard representation. A
representation of $\GL$ is \defn{polynomial} if it occurs as a
subquotient of a direct sum of tensor powers of $\bV$. Given a
polynomial representation $P$ of finite length in the abelian category
of polynomial representations, let $\bA(P)$ be the scheme $\Spec(\Sym(P))$, equipped with its natural $\GL$-action. An affine \defn{$\GL$-variety} is a reduced scheme equipped with an action of $\GL$ that admits a closed $\GL$-equivariant embedding into some $\bA(P)$. These are our main objects of study.

In \cite{polygeom}, we studied $\GL$-varieties in characteristic~0. We
defined a $\GL$-variety to be \defn{elementary} if it has the form $B
\times \bA(P)$, where $B$ is a finite dimensional variety (with
trivial $\GL$-action). We also define a notion of elementary
morphisms, and local versions of these concepts. The main theme of
\cite{polygeom} is that these elementary objects are the building
blocks of general objects. A bit more precisely, one of the main
results of \cite{polygeom} is the decomposition theorem, which says
that if $f \colon Y \to X$ is a map of $\GL$-varieties, then $Y$ and $X$ can be stratified by locally elementary varieties such that $f$ induces locally elementary maps on the strata. This is a very powerful result, and has immediate applications, such as a version of Chevalley's theorem for $\GL$-varieties.

In this paper, we follow a similar plan. In particular, one of our main results is a version of the decomposition theorem in positive characteristic (Theorem~\ref{thm:decomp}). However, there are two main
difficulties we have to overcome. First, in positive characteristic
there are interesting ``twisted'' affine spaces, and we must allow these
in our class of elementary varieties. And second, many statements only
hold ``up to Frobenius.'' For example, in our notion of locally elementary
maps we allow for $F$-equivalences (see \S \ref{ss:frob}).

\subsection{General coefficient rings} \label{ss:genco}

It is not difficult to extend the results of this paper to
$\GL$-varieties over arbitrary noetherian base rings. The basic idea
is the following. Suppose $X$ is a $\GL$-variety over a noetherian
ring $k$, and one wants to prove the decomposition theorem in this
setting. First, reduce to the case of a domain, which is elementary.
Then consider the base change of $X$ to $\Frac(k)$. The decomposition
theorem holds there by the results of this paper. One then spreads out
the decomposition to $k[1/f]$ for some non-zero $f$. Finally, one
passes to $k/(f)$ and runs the same argument. The process terminates
since $\Spec(k)$ is noetherian. There are some technicalities in this
process, e.g.~in the definition of polynomial functors; but these are
dealt with in \cite{BDD}. 

\subsection{Module theory}

In this paper and its predecessor \cite{polygeom}, we study $\GL$-algebras geometrically. One can also study these algebras algebraically, through their module theory. There has been much work on this in characteristic~0, e.g., \cite{symc1, symu1, Snowden}. Only recently have the first serious results been obtained in positive characteristic \cite{Ganapathy1,Ganapathy2}. We hope to see more work in this direction in the future.

\subsection{Outline}

In \S \ref{s:bg}, we recall some background material. In \S
\ref{s:torsor}, we discuss torsors for vector groups; this material
plays an important role in \S \ref{s:embed}, where we prove the
embedding theorem. In \S \ref{s:elem}, we introduce elementary
varieties and maps, and establish numerous basic results about them.
In \S \ref{s:shift}, we prove our main structural theorems for
$\GL$-varieties---the shift, unirationality, and decomposition
theorems---and give some consequences. Finally, in \S
\ref{s:strength}, we prove our application to strength (Theorem~\ref{thm:astr2}).

\subsection{Notation and terminology}

Some terminology and notation:
\begin{itemize}
\item A ``characteristic power'' is~1 if $\chr(K)=0$, or a power of $p$ if $\chr(K)=p>0$.
\item We prefix certain properties with the symbol $F$ to indicate that they hold ``up to Frobenius.'' For instance, a ring map is $F$-injective if every element of the kernel is killed by a power of Frobenius. We use the term $F$-equivalence for a map that is an isomorphism up to Frobenius. See \S \ref{ss:frob}.
\item A \defn{prevariety} is a scheme that is of finite type over $K$.
We similarly define \defn{$\GL$-prevariety}. 
\item A \defn{simple open} of an affine $\GL$-variety $X$ is an open set of the form $h \ne 0$, where $h$ is a $\GL$-invariant function on $X$. 
\item We use $\otimes_R$ to denote tensor products over a
(commutative, unital) ring $R$, and usually abbreviate $\otimes_R$ to
$\otimes$. Similarly, we use $\Hom_R(-,-)$ to denote the $R$-module of
$R$-module homomorphisms, and usually abbreviate $\Hom_K$ to $\Hom$.
\end{itemize}

We list a few important pieces of notation:
\begin{description}[align=right,labelwidth=2cm,leftmargin=!]
\item [$K$] the base field;
\item [$p$] the characteristic of $K$, either~0 or a prime number;
\item [$q$] a characteristic power (1 if $p=0$, and a power of $p$ if
$p$ is prime);
\item [$G(n)$] the subgroup $\GL_{\infty-n}$ of $\GL$ (\S \ref{ss:polyrep})
\item [$\bA(P)$] the $\GL$-variety associated to the polynomial
representation $P$;
\item [$\Sym^d(V)$] the $d$-th symmetric power of $V$; and 
\item [$\Sym(V)$] the symmetric algebra $\bigoplus_{d=0}^\infty
\Sym^d(V)$ on $V$.
\end{description}

\subsection*{Acknowledgments}

The third author thanks Karthik Ganapathy for helpful discussions.

\section{Background} \label{s:bg}

In this section, we collect various pieces of background material: we recall material on polynomial functors and representations, define $\GL$-algebras and $\GL$-varieties, and discuss some aspects of the Frobenius map in positive characteristic.

\subsection{Polynomial functors} \label{ss:poly}

Let $\Vec_K$ denote the category of all $K$-vector spaces. Covariant functors $\Vec_K \to \Vec_K$ form an abelian category, in which the morphisms are natural transformations. We write $P\{V\}$ for the value of a functor on a vector space $V$. An example of such a functor is the $d$th tensor power functor $\mathbb{T}^d$, defined by $\mathbb{T}^d\{V\}=V^{\otimes d}$. We define a \defn{naive polynomial functor} to be a functor $\Vec_K \to \Vec_K$ that is isomorphic to a subquotient of a (possibly infinite) direct sum $\bigoplus_{i \in I} \mathbb{T}^{d_i}$. If $K$ is infinite, naive polynomial functors are exactly the objects we want to use. However, if $K$ is finite then we will require a more sophisticated notion. The purpose of \S \ref{ss:poly} is to explain this idea.

We first recall the type of problem that occurs when $K$ is finite.
Essentially, we want polynomial functors to be compatible with
extending the coefficient field $K$. It turns out that naive
polynomial functors automatically satisfy this when $K$ is infinite,
but not when $K$ is finite. For instance, if $K=\bF_2$ then the
identity functor $V \mapsto V$ is a subobject of the symmetric square
$V \mapsto \Sym^2 V$ via the natural transformation $V \mapsto \Sym
V,\ v \mapsto v^2$, but this does not hold over any proper extension
field of $\bF_2$.

The category of finite dimensional vector spaces is enriched over the category of varieties: this means that if $U$ and $V$ are finite dimensional vector spaces then $\Hom(U,V)$ is an algebraic variety (it is simply an affine space). One way to enforce the compatibility with base change, in this setting, is to work with enriched functors. Such a functor $P$ assigns to vector spaces $V$ and $U$ a map of varieties
\begin{displaymath}
\Hom(U,V) \to \Hom(P\{U\}, P\{V\}),
\end{displaymath}
whereas an ordinary functor only gives a function on the $K$-points of these varieties. This leads to the notion of a \defn{strict polynomial functor}, first introduced in \cite{FS}.

We will need to work with functors that are both defined on infinite
dimensional spaces, and take values in infinite dimensional spaces
(even on finite dimensional input). The enriched functor approach becomes
technically more difficult in this setting, and so we give an alternative
definition. This definition is inspired by, but slightly different from,
the definition from \cite{Touze} via polynomial laws. In what follows,
we let $\Mod_R$ denote the category of $R$-modules and $\Free_R$ the
category of free $R$-modules (of possibly infinite rank).

\begin{definition}
A \defn{polynomial functor} $P$ over $K$ consists of the following data:
\begin{itemize}
\item for each $K$-algebra $R$, a covariant functor $P_R \colon \Free_R \to \Mod_R$; and
\item for each $K$-algebra homomorphisms $R \to S$, a natural isomorphism
\begin{displaymath}
\iota_{R \to S} \colon S \otimes_R P_R(-) \to P_S(S \otimes_R -)
\end{displaymath}
of functors $\Free_R \to \Mod_S$.
\end{itemize}
There is one axiom we impose on this data: given $K$-algebra homomorphisms $R \to S \to T$
and $V \in \Free_R$ the following diagram must commute:
\begin{displaymath}
\xymatrix{
T \otimes_{S} (S \otimes_R P_R\{V\}) \ar[rrrr]^{\id_{T} \otimes
\iota_{R \to S}} \ar[d]_{\cong} &&&& T \otimes_{S} P_{S}\{S
\otimes_R V\} \ar[d]^{\iota_{S \to T}}\\
T \otimes_R P_R\{V\} \ar[rr]_{\iota_{R \to T}} &&
P_{T}\{T \otimes_R V\} 
&& P_{T}\{T
\otimes_S (S \otimes_R V)\}. \ar[ll]^{\cong}
} \\[-2\dp\strutbox]
\end{displaymath}
A \defn{morphism} $P \to Q$ of polynomial functors is given
by a family of natural transformations $(\phi_R \colon P_R \to Q_R)_R$ such that
for all homomorphisms of $K$-algebras $R \to S$ and all $V \in \Free_R$
the following diagram commutes (by abuse of notation, we do not make
the dependence of $\iota$ on $P,Q$ explicit):
\[ \xymatrix{
S \otimes_R P_R\{V\} \ar[r]^{\iota_{R \to S}} \ar[d]_{\id_S \otimes
\phi_R} & P_S(S \otimes V) \ar[d]^{\phi_S} \\
S \otimes_R Q_R\{V\} \ar[r]^{\iota_{R \to S}} & Q_S(S \otimes V). }
\]
In this way, polynomial functors over $K$ form a category. One easily sees (using Remark~\ref{rmk:poly-free}) that is abelian, and carries a natural $K$-linear structure.
\end{definition}

\begin{example}
The prototypical example of a polynomial functor is the $d$-th tensor power $\bT^d$: here $\bT^d_R\{V\}=V^{\otimes d}$, and $\iota_{R \to S} \colon S \otimes_R V^{\otimes d} \to (S \otimes_R V)^{\otimes d}$ is the natural isomorphism. In each case, the tensor power is taken over the natural ring.
\end{example}

\begin{remark} \label{rmk:poly-free}
Suppose $P$ is a polynomial functor and $V$ is a free $R$-module. Then $V \cong R \otimes_K V_0$ for some $K$-vector space $V_0$. By compatibility with base change, we have $P_R(V) \cong R \otimes_K P_K(V_0)$, and so $P_R(V)$ is a free $R$-module. We thus see that $P_R$ necessarily takes values in the category $\Free_R$ of free $R$-modules.
\end{remark}

Let $P$ be a polynomial functor and let $x \in P_R\{V\}$. We say that $x$ is \defn{homogeneous of degree $d$} if for every $K$-algebra homomorphism $\phi \colon R \to S$ and every $t \in S$, we have $P_S\{t \cdot \id\}(x_S)=t^d \cdot x_S$, where here $x_S=\iota_{R \to S}(1 \otimes x)$ and $\id$ is the identity map on $S \otimes_R V$. One easily sees that the set of all such elements forms a polynomial subfunctor $P_d$ of $P$.

We have the following proposition summarising the relation between polynomial 
functors in the sense above to strict polynomial functors.

\begin{proposition}
Let $P$ be a polynomial functor over $K$. Then:

\begin{enumerate}

\item $P$ is a subquotient of a (possibly infinite) direct sum $\bigoplus_{i \in I} \bT^{d_i}$ of tensor power functors.

\item $P = \bigoplus_{d \ge 0} P_d$ is the direct sum of its homogeneous components.

\item $P$ is a direct limit of finite length objects. If $P$ itself has
finite length, then $P_K$ sends finite dimensional $K$-spaces to finite
dimensional $K$-spaces.

\item If $P_K$ maps finite dimensional $K$-spaces to finite dimensional
$K$-spaces, then there exists a unique strict polynomial functor $Q$
in the sense of \cite{FS} that agrees with $P$ in the following sense:
$Q\{U\}=P_K\{U\}$ for every $U \in \Vec_K$, and for any $K$-algebra
$S$ and any $U,V \in \Vec_K$ and any $\psi \in S \otimes \Hom(U,V)
\cong \Hom_S(S \otimes U, S \otimes V)$
the following diagram commutes:
\[ 
\xymatrix{
S \otimes P_K\{U\}=S \otimes Q\{U\} \ar[rr]^{Q\{\psi\}}
\ar[d]_{\iota_{K \to S}} && S \otimes Q\{V\} = S
\otimes P_K\{V\} \ar[d]^{\iota_{K \to S}} \\
P_S\{S \otimes U\} \ar[rr]_{P_S\{\psi\}} && P_S\{S \otimes V\};
}
\]
here $Q\{\psi\}$ stands for the map on $S$-points induced by the
morphism $Q\{\psi\}$. 
The assignment $P \mapsto Q$ is an equivalence of abelian categories
between the full subcategory of polynomial functors in our sense that
map finite dimensional $K$-spaces to finite dimensional $K$-spaces and
the category of strict polynomial functors.
\end{enumerate}
\end{proposition}

\begin{proof}[Proof sketch.]
We only sketch the constructions for (d).  Let $U,V$ be finite dimensional
$K$-spaces. After choosing bases, we identify $U$ with $K^n$ and $V$ with $K^m$. Write $R:=K[x_{ij} \mid
i \in [m], j \in [n]]$, where the $x_{ij}$ are variables. Let $\phi \in
\Hom_R(R \otimes U, R \otimes V)=\Hom_R(R^n,R^m)$ be the matrix with
$(i,j)$-entry $x_{ij}$. Choose $K$-bases of $P_K\{U\}$ and $P_K\{V\}$, and
call their sizes $n'$ and $m'$, respectively. These bases map to bases
of $P_R\{R \otimes U\}$ and $P_R\{R \otimes V\}$ via $\iota_{K \to R}$. With respect to these
bases, the $R$-module homomorphism $P\{\phi\}$ is a matrix with entries
from $R$, i.e., polynomials in the $x_{ij}$. This $m' \times n'$-tuple
of polynomials in $R$ defines a morphism
\[ Q: \Spec(\Sym(\Hom(U,V)^*)) \to
\Spec(\Sym(\Hom(P\{U\},P\{V\})^*)).
\]
A straightforward computation shows that this morphism does not depend
on the choice of bases of $U,V,P\{U\},P\{V\}$ and that this construction
with $U,V$ varying yields a strict polynomial functor $Q$. To show that
$Q$ agrees with $P$, let $S$ be another $K$-algebra and let $\psi \in
S \otimes \Hom(U,V)$.  Let $\psi_{ij}$ be the matrix
entries of $\psi$ with respect to the chosen bases and let $R \to S$
be the $K$-algebra homomorphism determined by $x_{ij} \mapsto
\psi_{ij}$. We claim that the following diagram commutes:
\[
\xymatrix{
S \otimes P_K\{U\} \ar@/_15ex/[dddd]_{\iota_{K \to S}} \ar[rr]^{Q\{\psi\}} 
&& S \otimes P_K\{V\} \ar@/^15ex/[dddd]^{\iota_{K \to S}} \\
S \otimes_R (R \otimes P_K\{U\}) \ar[u]^{\cong} \ar[d]_{\id_S \otimes
\iota_{K \to R}} && 
S \otimes_R (R \otimes P_K\{V\}) \ar[u]^{\cong} \ar[d]_{\id_S \otimes
\iota_{K \to R}}\\
S \otimes_R P_R\{R \otimes U\} \ar[d]_{\iota_{R \to S}} \ar[rr]^{\id_S
\otimes P_R\{\phi\}} &&
S \otimes_R P_R\{R \otimes V\} \ar[d]_{\iota_{R \to S}}  \\
P_S\{S \otimes_R (R \otimes U)\} \ar[d]_{\cong} \ar[rr]_{P\{\id_S \otimes
\phi\}} &&
P_S\{S \otimes_R (R \otimes V)\} \ar[d]_{\cong} \\
P_S\{S \otimes U\} \ar[rr]_{P_S\{\psi\}} && P_S\{S \otimes V\}.} \]
Indeed, the left semicircle commutes by our axiom for $\iota$, and so
does the right semicircle. The center square commutes because $\iota_{R
\to S}$ is a natural transformation. The bottom square commutes because
of our choice of homomorphism $R \to S$, and the top rectangle commutes
by our definition of $Q$. This shows that $Q$ agrees with $P$. 

To show that $P \mapsto Q$ is essentially surjective, 
let $Q$ be a strict polynomial functor. We construct
a polynomial functor $P$ in our sense as follows. We first define
$P=(P_R)_R$ on finite rank modules as follows. For any $K$-algebra $R$
and any free $R$-module $U$ of finite rank choose a $K$-vector space $U'$
and an isomorphism $R \otimes U' \to U$. Then set $P_R\{U\}:=R \otimes
Q\{U'\}$. For any homomorphism $\phi: U \to V$ of free $R$-modules
of finite rank, we have a unique $R$-module $\phi':R \otimes U' \to R
\otimes V'$ such that the following diagram commutes:
\[ 
\xymatrix{
R \otimes U' \ar[d]_{\phi'} \ar[r]^{\cong} & U \ar[d]^{\phi} \\
R \otimes V' \ar[r]_{\cong} & V 
}.
\]
Now $\phi'$ is an $R$-point of $\Spec(\Sym(\Hom(U',V')^*))$, and hence
$Q$ associates an $R$-point of $\Spec(\Sym(\Hom(Q\{U'\},Q\{V'\}))^*)$
to it, i.e., an $R$-module homomorphism $Q\{\phi'\}: R \otimes Q\{U'\}
\to R \otimes Q\{V'\}$. We define $P_R\{\phi\}:=Q\{\phi'\}$. In the
general case, where $U$ is not necessarily finite dimensional, we can
define $P(U)$ as the direct limit $\varinjlim P(M)$ where $M$ runs over
finitely generated free $R$-submodules of $M$, and similarly for
morphisms. 

Next, we need to exhibit the natural transformation $\iota_{R \to S}$
for any $K$-algebra homomorphism $R \to S$. Again, we first define
$\iota_{R \to S}: S \otimes P_R\{U\} \to P_S\{S \otimes U\}$ 
for a finitely generated free $R$-module $U$. Set $V:=S \otimes_R U$,
and recall that we have also fixed an isomorphism $S \otimes V' \to V$ where
$V'$ is a finite dimensional $K$-space. Now there is a unique
$S$-module isomorphism $\psi:S \otimes U' \to S \otimes V'$ that makes
the following diagram commute:
\[ 
\xymatrix{
S \otimes_R (R \otimes U') \ar[rr]^{\id_S \otimes \cong} \ar[d]_{\cong}
&& S \otimes_R U=V \\
S \otimes U' \ar[rr]_{\psi} && S \otimes V' \ar[u]_{\cong}
}
\]
and the strict polynomial functor $Q$ assigns to $\psi$ an $S$-module
homomorphism 
\[ Q\{\psi\}:P_S\{U\}=S \otimes Q\{U'\} \to S \otimes Q\{V'\}
=P(S \otimes V'). \] 
We define $\iota_{R \to S}$ at $U$ as $Q\{\psi\}$. We extend
$\iota_{R \to S}$ to arbitrary free $R$-modules by a limit. To prove
that $\iota$ is well-defined and has the desired properties is tedious
but straightforward.
\end{proof}

\begin{corollary}
If $K$ is infinite then $P \mapsto P_K$ is an equivalence between the category of polynomial functors and the category of naive polynomial functors.
\end{corollary}

\begin{proof}
This is well-known for strict polynomial functors. The result thus follows since all flavors of polynomial functors are direct limits of finite length objects.
\end{proof}

From now on, we will usually evaluate a polynomial functor $P$ only at
$K$-vector spaces, and write $P\{V\}$ rather than $P_K\{V\}$.  We will
only carry out constructions that are in a straightforward manner
compatible with base change, so we will not explicitly indicate the
natural isomorphisms $\iota_{R \to S}$. This will not lead to confusion.

\subsection{Operations on polynomial functors}

Let $U$ be a finite dimensional vector space and let $P$ be a polynomial
functor. The \defn{shift of $P$ over $U$}, denoted $\Sh_U(P)$, is the
polynomial functor that assigns to $V \in \Vec_K$ the space $P\{U \oplus V\}$ and to
$\phi \in \Hom_K(V,W)$ the map $P\{\id_U \oplus \phi\}$ (enriched 
in the obvious manner to be compatible with base change to $K$-algebras).
We will also use the notation $\Sh_n(P)$ for $\Sh_{K^n}(P)$.

Assume $\chr(K)=p$ and let $q$ be a characteristic power. An important operation
on polynomial functors is the {\em Frobenius twist}: $P^{(q)}$ assigns to
$V$ the space $K \otimes_K P\{V\}$, where the $K$-module structure of $K$ comes from the $q$-th
power map $K \to K$. If $P$ has finite degree $d$, then $P^{(q)}$ is a polynomial
functor of finite degree $qd$.

A last ingredient in our proofs is the following partial order on
polynomial functors: we say that $Q$ is smaller than $P$ if they are not
isomorphic and for the smallest $e$ with $Q_e$ not isomorphic to $P_e$,
$Q_e$ is a sub-object of $P_e$. This is a well-founded partial order on
(isomorphism classes of) polynomial functors. A typical instance of
this that we will use is the following: if $P$ has degree $d$ and $U$
is a finite dimensional vector space, then the shift $\Sh_U P$ also has
degree $d$ and satisfies $(\Sh_U P)_d \cong P_d$. So if in $\Sh_U P$
we replace the top-degree part by a proper subobject, then we obtain a
polynomial functor strictly smaller than $P$.

\subsection{Polynomial representations} \label{ss:polyrep}

The space $\bV=\bigcup_{n \ge 1} K^n$ is a representation of the ind-group-scheme $\GL=\bigcup_{n \ge 1} \GL_n$. A \defn{polynomial representation} of $\GL$ is one that occurs as a subquotient of a possible infinite direct sum of tensor powers of $\bV$. If $K$ is infinite then one can replace the ind-group-scheme $\GL$ with the discrete group $\GL(K)$.

Given a polynomial functor $P$, the group $\GL(K)$ acts on $P\{\bV\}=\bigcup_n P\{K^n\}$ via the group homomorphism $g \mapsto P\{g\}$. The compatibility of $P$ with base change yields an action of $\GL$. The assignment $P \mapsto P\{\bV\}$ is an equivalence of categories from polynomial functors to polynomial
representations of $\GL$.

The shift and Frobenius twist operations on polynomial functors translate to similar operations on polynomial representations. If $P$ is a polynomial representation then $\Sh_n(P)$ can
be regarded naturally as the polynomial representation obtained from $P$
by precomposing with the injective ind-group endomorphism $\GL \to \GL,\
g \mapsto \diag(1_n,g)$, where $1_n$ stands for the $n \times n$-identity
matrix. We write $G(n)$ for the image of this endomorphism.

Suppose $P$ is a polynomial representation and $x$ is an element of $P$. If the field $K$ is finite then the $\GL(K)$-subrepresentation generated by $x$ and the $\GL$-subrepresentation generated by $x$ may not coincide. The following proposition gives us some control on this situation.

\begin{proposition} \label{prop:FiniteExt}
Let $P$ be a polynomial $\GL$-representation, let $x_1,\ldots,x_k \in
P=P\{\bV\}$, and let $Q$ be the smallest polynomial subrepresentation of
$P$ containing the $x_i$. Then there exists a finite extension $L$ of $K$
such that $L \otimes Q$ equals the $L$-linear span of the orbits $\GL(L)
\cdot (1 \otimes x_i),\ i=1,\ldots,k$.
\end{proposition}

\begin{proof}
This follows from the proof of \cite[Lemma 2.7]{FS}---there it is shown
that any infinite field $L$ containing $K$ has this property, but the
proof shows that one can choose $L/K$ of finite degree that can be
controlled by the polynomial representation $Q$.
\end{proof}

\subsection{$\GL$-varieties} \label{ss:glvar}

All our schemes will be over $K$. A {\em pre-variety} is a scheme of
finite type over $K$, and a {\em variety} is a reduced pre-variety. Note
that a variety may not stay reduced under base change to $\ol{K}$: for
$K=\bF_p(t)$ the scheme $\Spec(\bF_p(t)[x]/(x^p-t))$ is reduced but its
base change to $\ol{K}$ is not. The class of pre-varieties is preserved
under base change.

A {\em $\GL$-algebra} is a (commutative, unital) algebra object in
the category of polynomial representations of $\GL$. It is called
{\em finitely $\GL$-generated} if is generated, as an algebra, by a
finite-length polynomial subrepresentation. 

An affine {\em $\GL$-scheme} $X$ is the spectrum of a $\GL$-algebra
$B$, and a {\em closed $\GL$-subscheme} of $X$ is the spectrum of a
$\GL$-algebra of the form $B/I$, where $I$ is an ideal in $B$ as well
as a subobject of $B$ in the category of polynomial representations
over $K$. An {\em open $\GL$-subscheme} of $X$ is the complement of a
closed $\GL$-subscheme, with the usual scheme structure. A {\em locally
closed $\GL$-subscheme} of $X$ is an open $\GL$-subscheme within a closed
$\GL$-subscheme of $X$.

Note that $X$ comes with an algebraic action of the ind-group $\GL$ by
automorphisms. The closed $\GL$-subschemes of $X$ are precisely the closed
subschemes that are stable under this ind-group action (this requirement
is stronger than that they are stable under the abstract group $\GL$).

An affine {\em $\GL$-prevariety} is an affine scheme of the form
$\Spec(R)$ where $R$ is a finitely $\GL$-generated algebra. An affine
{\em $\GL$-variety} is a reduced $\GL$-prevariety. We will also
encounter {\em quasi-affine} $\GL$-(pre-)varieties: these are open
$\GL$-sub(pre-)varieties in affine $\GL$-(pre-)varieties. The class of
$\GL$-prevarieties is preserved under base change to $\ol{K}$, fiber
products, and Frobenius twist.

Note that $\GL$-algebras $R$ (respectively, $\GL$-schemes $X$) are
naturally covariant (respectively, contravariant) functors from vector
spaces to algebras (respectively, schemes) over $K$. We will use the
notation $R\{U\}$ and $X\{U\}$, in accordance with that for polynomial
functors. We will sometimes write $R_0:=R\{0\}$ and $X_0:=X\{0\}$.  Also,
the shifts $\Sh_U(R), \Sh_n(R)$ and $\Sh_U(X), \Sh_n(X)$ are defined in the
natural manner.

The prototypical examples of $\GL$-varieties are those of the form
$\bA(P):=\Spec(\Sym(P))$, where $P$ is a finite-length polynomial
representation. These play the role of affine spaces in our theory.
For instance, every $\GL$-prevariety $X$ is a closed $\GL$-subscheme of some such space. Indeed, $X=\Spec(R)$ for some is a finitely
$\GL$-generated algebra $R$. Let $P \subset R$ be a finite-length
$\GL$-subrepresentation that generates $R$ as an algebra. Then we have
a natural surjection $\Sym(P) \to R$, which dually gives the desired
closed embedding $X \to \bA(P)$.

Although $\GL$-varieties are almost always infinite dimensional, they nonetheless satisfy an important noetherianity property \cite{draisma}:

\begin{theorem} \label{thm:noeth}
Let $X$ be a $\GL$-variety. Then any descending chain of $\GL$-stable closed subsets of $X$ stabilizes.
\end{theorem}

Let $X$ be a quasi-affine $\GL$-scheme. A \defn{simple open} is an open subset of the form $X[1/h]$, where $h$ is a $\GL$-invariant function on $X$. Such a set is itself an affine $\GL$-scheme. The intersection of two simple opens is still a simple open: we have $X[1/h] \cap X[1/h'] = X[1/hh']$. If $f \colon Y \to X$ is a map of affine $\GL$-schemes then the inverse image of a simple open is a simple open; indeed, $f^{-1}(X[1/h])=Y[1/f^*(h)]$.

Given an arbitrary open subset $U$ of $X$, we write $\GL \cdot U$ for the smallest $\GL$-stable open subset of $X$ containing $U$. If $K$ is infinite then $\GL \cdot U$ is the union of the translates $gU$ of $U$ with $g \in \GL(K)$; if $K$ is finite, the same is true provided we use $\GL(\ol{K})$ instead of $\GL(K)$. In fact, we sometimes have an even stronger statement:

\begin{proposition} \label{prop:GLU}
In the above setting, suppose that $U$ is $G(n)$-stable for some $n$. Then there exists a finite extension $\Omega/K$ such that $\GL \cdot U$ is the union of the sets $gU$ with $g \in \GL(\Omega)$.
\end{proposition}

\begin{proof}
It suffices to treat the case where $X$ is affine (replace $X$ with the spectrum of its coordinate ring). Let $R$ be the coordinate ring of $X$. By noetherianity (Theorem~\ref{thm:noeth}), the closed
$G(n)$-subscheme $X \setminus U$ of $X$ is defined by the radical
of a finitely $G(n)$-generated ideal $I$ of $R$. Let $f_1,\ldots,f_k$
be generators of $I$. Then $X$ is defined by the (radical of the ideal
generated by) smallest polynomial $\GL$-subrepresentation $Q$ of $R$
containing $f_1,\ldots,f_k$. By Proposition~\ref{prop:FiniteExt}, there
exists a finite extension $\Omega$ of $K$ such that $\Omega \otimes Q$ is generated
by the orbits $\GL(\Omega) \cdot (1 \otimes f_i),\ i=1,\ldots,k$. This
extension has the desired property.
\end{proof}

\subsection{Frobenius} \label{ss:frob}

Let $f \colon A \to B$ be a homomorphism of $K$-algebras. We say that $f$
is \defn{$F$-injective} if for every element $x$ of the kernel there
is some characteristic power $q$ such that $x^q=0$. We say that $f$
is \defn{$F$-surjective} if for every element $x \in B$ there is
some characteristic power $q$ such that $x^q \in \im(f)$.  We say
that $f$ is an \defn{$F$-equivalence} if it is both $F$-injective and
$F$-surjective. Given $f \colon A \to B$ and $g \colon B \to C$ we have analogues of familiar statements about
isomorphisms, e.g.: if $f$ (respectively, $g$) is an $F$-equivalence,
then $g \circ f$ is one if and only if $g$ (respectively, $f$) is. We say that a map of affine schemes $f \colon Y \to X$ is an \defn{$F$-equivalence} if the map on rings is, and a \defn{closed $F$-immersion} if the map on rings is $F$-surjective.

We make two comments about the above definitions. First, in characteristic~0, $F$-injective is equivalent to injective, $F$-surjective is equivalent to surjective, and $F$-equivalence is equivalent to isomorphism. Second, these properties are actually independent of the $K$-algebra structure.

\begin{proposition} \label{prop:fg-F-equiv}
Let $f \colon A \to B$ be a $\bF_p$-algebra homomorphism, and let $F$
denote the absolute Frobenius map (i.e., $F(x)=x^p$) on either of these $\bF_p$-algebras.
\begin{enumerate}
\item Suppose there is a ring homomorphism $g \colon B \to A$ and $n
\ge 0$ such that $g \circ f=F^n$ and $f \circ g = F^n$. Then $f$ is an $F$-equivalence.
\item Suppose $f$ is an $F$-equivalence and $A$ and $B$ are finitely generated $\bF_p$-algebras. Then a map $g$ as in (a) exists.
\end{enumerate}
\end{proposition}

\begin{proof}
(a) Put $q=p^n$. If $f(x)=0$ then $0=g(f(x))=x^q$, and so $f$ is $F$-injective. Given $y \in B$, we have $y^q=f(x)$ where $x=g(y)$, and so $f$ is $F$-surjective.

(b) Let $r \ge 0$ be such that $x^{p^r}=0$ for every nilpotent $x \in A$, and let $s \ge 0$ be such that $y^{p^s} \in \im(f)$ for all $y \in B$; such $r$ and $s$ exists since $A$ and $B$ are finitely generated. Put $n=r+s$. Define $g$ to be the following composition
\begin{displaymath}
\xymatrix{
B \ar[r]^-{F^s} & \im(f) \ar[r] & A/\rad(A) \ar[r]^-{F^r} & A. }
\end{displaymath}
Note that since $f$ is $F$-injective, we have $\ker(f) \subset \rad(A)$, which is why the second map above exists. Explicitly, if $y \in B$ is given then $y^{p^s}=f(x)$ for some $x \in A$, and $g(y)$ is defined to be $x^{p^r}$. We thus have $f(g(y))=f(x)^{p^r}=y^{p^n}$. Also, if $x \in A$ is given and $y=f(x)$ then $y^{p^s}=f(x^{p^s})$, and so $g(f(x))=g(y)=x^{p^n}$. The result follows.
\end{proof}

\begin{proposition} \label{prop:F-equiv-bc}
Let $f \colon A \to B$ and $A \to A'$ be homomorphisms of $\bF_p$-algebras, and let $f' \colon A' \to B'$ be the base change of $f$. If $f$ is an $F$-equivalence (resp.\ $F$-surjection) so is $f'$.
\end{proposition}

\begin{proof}
First suppose that $f$ is an $F$-surjection. Let $u=x \otimes y$ in $B'=A' \otimes_A B$ be given. We have $y^q=f(z)$ for some characteristic power $q$ and some $z \in A$. Thus $u^q=(zx^q) \otimes 1$ belongs to the image of $f'$. Since pure tensors span $B'$, it follows that $f'$ is an $F$-surjection.

Now suppose that $f$ is an $F$-equivalence. First suppose that $A$ and $B$ are finitely generated $\bF_p$-algebras. Let $g \colon B \to A$ be a ring homomorphism such that $g \circ f = F^n$ and $f \circ g = F^n$, which exists by Proposition~\ref{prop:fg-F-equiv}(b). Let $g' \colon B' \to A'$ be the unique ring homomorphism such that the diagram
\begin{displaymath}
\xymatrix@R=4pt@C=8ex{
& B \ar[rd] \ar@/^12pt/[rrd]^g \\
A \ar[ru]^f \ar[rd] && B' \ar@{..>}[r]^-{g'} & A' \\
& A' \ar[ru]^{f'} \ar@/_12pt/[rru]_{F^n} }
\end{displaymath}
commutes, where the top map $g$ is really the composition of $g$ with the natural map $A \to A'$. The map $g'$ exists by the mapping property for $B'$. We have $g' \circ f'=F^n$ by definition. Writing $B'=A' \otimes_A B$, we have $g'(x \otimes y)=F^n(x) \cdot g(y)$ and $f'(x)=x \otimes 1$. Thus
\begin{displaymath}
f'(g'(x \otimes y)) = (F^n(x) g(y)) \otimes 1 = F^n(x) \otimes f(g(y)) = F^n(x) \otimes F^n(y) = F^n(x \otimes y),
\end{displaymath}
and so $f' \circ g'=F^n$. It follows from Proposition~\ref{prop:fg-F-equiv}(a) that $f'$ is an $F$-equivalence.

We now treat the general case. Write $A=\bigcup_{i \in I_1} A_i$ and $B=\bigcup_{i \in I_2} B_i$, where both unions are directed, and $A_i$ and $B_i$ are finitely generated $\bF_p$-algebras. Let $I$ be the set of pairs $(i_1,i_2)$ such that $f$ induces an $F$-equivalence $A_{i_1} \to B_{i_2}$. For $i=(i_1,i_2)$, put $A_i=A_{i_1}$, $B_i=B_{i_2}$, and let $f_i$ be the map induced by $f$. One easily sees $A=\bigcup_{i \in I} A_i$ and $B=\bigcup_{i \in I} B_i$.

Let $f'_i \colon A' \to B'_i$ be the base change of $f_i$ along the map $A_i \to A'$. This is an $F$-equivalence by the first paragraph. Since tensor products are compatible with direct limits, we have $f'=\varinjlim f'_i$ and $B'=\varinjlim B'_i$. It now follows easily that $f'$ is an $F$-equivalence. Indeed, suppose $x \in \ker(f')$. Then $f'_i(x)=0$ in $B'_i$ for some $i$, and so $x$ is nilpotent since $f'_i$ is an $F$-equivalence. Thus $f$ is $F$-injective. Now let $y \in B'$ be given. Then $y$ comes from some $y' \in B'_i$ for some $i$. Since $f'_i$ is an $F$-equivalence, we have $(y')^q \in \im(f'_i)$ for some $i$, and so it follows that $y^q \in \im(f')$. Thus $f'$ is $F$-surjective.
\end{proof}

Suppose that $A \to B$ is a $K$-algebra homomorphism, and let $q$ be
a characteristic power. We define the \defn{Frobenius twist} of $B$
relative to $A$, denoted $B^{(q)}$, to be $A \otimes_A B$, where the $A$-module structure of $A$ comes from the $q$-th power map $A \to A$. There is a natural $A$-algebra
homomorphism $B^{(q)} \to B$ given by $a \otimes b \mapsto ab^q$, which
is called the \defn{Frobenius map} for $B$ relative to $A$. It is easy
to see that this is an $F$-equivalence. 

In almost all cases in this paper, we take Frobenius maps relative
to the ground field $K$. We warn the reader that, even in this case,
Frobenius twisting can introduce nilpotents. These constructions apply
to affine schemes as well.

Let $A$ be a $K$-algebra, let $B,C$ be $A$-algebras, and let $f:B
\to C$ be an $A$-algebra homomorphism. Then $f$ induces an $A$-algebra
homomorphism $f^{(q)}:B^{(q)} \to C^{(q)}, a \otimes b \mapsto a \otimes
f(b)$ between the Frobenius twists of $B$ and $C$ relative to $A$.

\begin{proposition} \label{prop:frob-equiv}
If an $A$-algebra homomorphism $f:B \to C$ is $F$-surjective 
(respectively, $F$-injective), then so is the induced homomorphism
$f^{(q)}:B^{(q)} \to C^{(q)}$.
\end{proposition}

\begin{proof}
First assume that $f$ is $F$-surjective and consider an element $u=\sum_i
a_i \otimes c_i \in C^{(q)}$. For each $i$ there exists a characteristic
power $q_i$ such that $c_i^{q_i}$ lies in the image of $f$. Let
$\tilde{q}$ be the maximum of the $q_i$; then $c_i^{\tilde{q}}=f(b_i)$
for some $b_i \in B$. We then have
\[ u^{\tilde{q}}=\sum_i a_i^{\tilde{q}} \otimes f(b_i) 
= f^{(q)} \left(\sum_i a_i^{\tilde{q}} \otimes b_i \right); \]
so $f^{(q)}$ is $F$-surjective.

Next assume that $f:B \to C$ is $F$-injective and let $u=\sum_i a_i
\otimes b_i \in B^{(q)}$ satisfy $f^{(q)}(u)=0$. Applying the
Frobenius map $C^{(q)} \to C$ to this identity, we find that 
\[ 0 = \sum_i a_i f(b_i)^q = f\left(\sum_i a_i b_i^q\right) \]
By $F$-injectivity of $f$, there exists a characteristic
power $\tilde{q}$ such that $(\sum_i a_i b_i^q)^{\tilde{q}}=0$. The
left-hand side is the image of $u^{\tilde{q}}$ under the Frobenius map
$B^{(q)} \to B$. Since this Frobenius map is an $F$-equivalence, it
follows that some characteristic power of $u^{\tilde{q}}$, hence of
$u$, is zero. Hence $f^{(q)}$ is $F$-injective. 
\end{proof}

All of the above applies in the $\GL$-context too. Note in particular
that if $B$ is a $\GL$-algebra then so is the Frobenius twist $B^{(q)}$
of $B$ relative to $K$, and the Frobenius map $B^{(q)} \to B$ is a map
of $\GL$-algebras. We saw in Proposition~\ref{prop:fg-F-equiv} that
for an $F$-equivalence $f$ of finitely generated algebras, a suitable
power of Frobenius factors via $f$. We will require the following
analog of this in the $\GL$-context. In the $\GL$-setting, we do not
know that finitely $\GL$-generated algebras are noetherian---the best
we know so far is that they are noetherian ``up to radicals'' by
\cite{draisma}---so we get a weaker statement.

\begin{proposition} \label{prop:frob-factor}
Let $A$ and $B$ be finitely $\GL$-generated $\GL$-algebras with $A$
reduced, and let $f \colon A \to B$ be a map of $\GL$-algebras that is
an $F$-equivalence. Then there exists a characteristic power $q$ and a
map $g \colon B^{(q)} \to A$ of $\GL$-algebras such that $f \circ g$
is the Frobenius on $B$ relative to $K$. 
\end{proposition}

\begin{proof}
Since $A$ is reduced, $f$ is actually injective, and so we regard it as a subalgebra of $B$. Since $f$ is $F$-surjective, for every element $x \in B$ there is some characteristic power $q$ such that $x^q \in A$. Since $B$ is finitely $\GL$-generated, there is a single $q$ that works for all $x$; we then have $B^q \subset A$. It follows that the Frobenius map $B^{(q)} \to B$ takes values in $A$, and provides the necessary factorization.
\end{proof}

\section{Torsors} \label{s:torsor}

Suppose we have a short exact sequence of polynomial representations
\begin{displaymath}
0 \to P_1 \to P_2 \to P_3 \to 0.
\end{displaymath}
We then obtain a map of $\GL$-varieties $\pi \colon \bA(P_2) \to \bA(P_1)$. How should we think of this map geometrically? For the purposes of this paper, it will be useful to view $\pi$ as an $\bA(P_3)$-torsor, meaning that the additive group $\bA(P_3)$ naturally has a simply transitive action on each fiber. This basic example leads to some less trivial examples: for example, if $X$ is a closed $\GL$-subvariety of $\bA(P_1)$ then $\pi^{-1}(X) \to X$ is also a $\bA(P_3)$-torsor. For this reason, torsors for vector groups will play an important role in our study of $\GL$-varieties. In \S \ref{s:torsor}, we develop the basic theory of these objects.

\subsection{Generalities}

Let $G$ be a group scheme over $K$ and let $X$ be a scheme over $K$. A \defn{$G$-scheme} over $X$ is a scheme $Y$ equipped with a map $\pi \colon Y \to X$ and an action of $G$ over $X$, that is, the action map $G \times Y \to Y$ is a map of schemes over $X$; this means that $G$ acts on the fibers of $\pi$. An example of a $G$-scheme over $X$ is $G \times X$; the $G$-action is on the first factor, and the map to $X$ is projection onto the second factor. If $Y$ is a $G$-scheme over $X$ and $X' \to X$ is any map, then the base change $Y'=Y \times_X X'$ is naturally a $G$-scheme over $X'$.

We say that a $G$-scheme $Y$ over $X$ is a \defn{$G$-torsor} if there is a Zariski open cover $X=\bigcup_{i \in I} U_i$ such that $Y \times_X U_i$ is isomorphic, as a $G$-scheme over $U_i$, to $G \times U_i$ for all $i \in I$. The \defn{trivial} $G$-torsor over $X$ is $G \times X$. Thus a torsor is, by definition, locally trivial\footnote{In general, asking for torsors to be trivial Zariski locally is too strong, and one should instead use a finer topology such as the fppf topology. For our purposes though, the Zariski topology will suffice.}. If $Y \to X$ is a $G$-torsor then its base change along a morphism $X' \to X$ is a $G$-torsor over $X'$.

Suppose that $\pi \colon Y \to X$ is a $G$-torsor. If $\pi$ admits a
section $s \colon X \to Y$ then $Y$ is isomorphic to the trivial
$G$-torsor: indeed, the map $G \times X \to Y$ defined by $(g,x)
\mapsto g s(x)$ is an isomorphism of $G$-torsors (one can check
this locally). If $s$ and $s'$ are two sections of $\pi$ then there is a morphism $g \colon X \to G$ such that $s'=gs$. Indeed, it suffices to prove this for the trivial torsor, and we can then take $g=s' \cdot s^{-1}$ (where here we consider $s$ and $s'$ to be morphisms $X \to G$).

\subsection{Vector groups} \label{ss:vecgp}

Let $V$ be a vector space over $K$, possibly of infinite dimension. Put $\bA(V)=\Spec(\Sym(V))$, which is an affine scheme whose
$K$-points form the dual space $V^*$. This is naturally a group scheme under addition. We now consider torsors for this group. We denote the action of $\bA(V)$ additively, since it will typically correspond to a translation action. We focus on the affine case, as that is what will be important in this paper.

Let $X=\Spec(A)$, and suppose that $Y=\Spec(B)$ is an $\bA(V)$-torsor over $X$. The action of $\bA(V)$ on $Y$ corresponds to a co-multiplication map
\begin{displaymath}
\Delta \colon B \to B \otimes \Sym(V).
\end{displaymath}
Given $f \in B$, we have a unique decomposition $\Delta(f)=\sum_{i \ge
0} \Delta_i(f)$, where $\Delta_i(f) \in B \otimes \Sym^i(V)$. We
define $B_{\le n}$ to be the set of elements $f \in B$ such that
$\Delta_i(f)=0$ for $i>n$; by convention, $B_{\le n}=0$ if $n<0$. Each
$B_{\le n}$ is an $A$-submodule of $B$, and we have $B_{\le 0}=A$
(this can be checked locally on $X$, where we can use a trivialization). We have $B_{\le n} \cdot B_{\le m} \subset B_{\le n+m}$, and so the $B_{\le n}$'s define a filtration on the ring $B$. There is a natural $A$-linear map
\begin{displaymath}
\init_n \colon B_{\le n} \to A \otimes \Sym^n(V),
\end{displaymath}
given by taking the initial term, as we now explain. If $Y$ is trivial
then, choosing a trivialization, we have $B_{\le n}=A \otimes
\Sym^{\le n}(V)$, and the above map takes the degree $n$ piece. This
construction is translation invariant, and thus independent of the
chosen trivialization. In general, we use this construction locally and glue. The map $\init_n$ clearly kills $B_{\le n-1}$, and so there is an induced map
\begin{displaymath}
\gr(B) = \bigoplus_{n \ge 0} B_n/B_{n-1} \to A \otimes \Sym(V).
\end{displaymath}
This map is a ring isomorphism (again, this can be checked locally).

\subsection{Derivatives} \label{ss:der}

Maintain the notation from \S \ref{ss:vecgp}. Given $f \in B$ and a $K$-point $r$ of $\bA(V)$ (i.e., an element of $V^*$), we define $\partial_r(f) \in B$ by applying $r$ to the second tensor factor of $\Delta_1(f) \in B \otimes V$. This is the directional derivative of $f$ in the $r$ direction. For $y \in Y$ and $t \in K$, we have
\begin{displaymath}
f(y+tr)=f(y)+t (\partial_r f)(y)+\cdots
\end{displaymath}
where the remaining terms have higher powers of $t$. If $\Delta_1(f)$ is non-zero then $\partial_r(f)$ is non-zero for some choice of $r$. We now investigate this condition more closely:

\begin{proposition} \label{prop:deriv}
Let $f$ be an element of $B$ that does not belong to $A$. Recall $p=\chr(K)$.
\begin{enumerate}
\item If $p = 0$ then $\Delta_1(f) \ne 0$.
\item If $p>0$ then there is a power $q$ of $p$ such that $\Delta_q(f)
\ne 0$ and $f \in A \cdot B^q$, i.e., $f$ is an $A$-linear combination of elements in $B^q$.
\end{enumerate}
\end{proposition}

\begin{proof}
First consider the case where $Y$ is trivial. Then we may assume that $B=A \otimes
\Sym(V)$ and the action of $\bA(V)$ on $Y$ is by translation on the second factor. We
may write $f=f(x) \in A[x_1,\ldots,x_n] \setminus A$ where the $x_i$
are linearly independent elements of $V$. Then $\Delta(f)=f(x+y) \in
A[x_1,\ldots,x_n,y_1,\ldots,y_n]$ and $\Delta_1(f)$ is the part of this
expression that is linear in $y$.

If $x_i$ appears in some monomial $x^\alpha$ with a nonzero coefficient $a
\in A$ in $f$, then the partial derivative $\frac{\partial f}{\partial
x_i}$ contains the term $\alpha_i \cdot a \cdot x^{\alpha-e_i}$. If
$\alpha_i \neq 0$ in $K$, then this term is nonzero and hence the
coefficient of $y_i$ in $\Delta_1(f)$, which equals this partial
derivative, is nonzero. This is the case, in particular, if $\chr(K)=0$.

If no such monomial with $\alpha_i \neq 0$ in $K$ exists for any
$i$, then $\chr(K)=p$ and $f(x)=\tilde{f}(x_1^p,\ldots,x_n^p)$ for some
$\tilde{f} \in A[x_1,\ldots,x_n]$ of total degree $\deg(f)/p$. Then
\[
\Delta(f)(x,y)=\Delta(\tilde{f})(x_1^p,\ldots,x_n^p,y_1^p,\ldots,y_n^p) \] 
and hence the
statement for $f$ follows from that for $\tilde{f}$. We are done by
induction on $\deg(f)$.

Now consider the general case: $X=\bigcup_{i \in I} U_i$, and each
$Y \times_X U_i$ is isomorphic, as an $\bA(V)$-scheme, to $U_i \times
\bA(V)$. We may further assume the $U_i$ to be simple open subsets of
$A$. Let $A_i=A[1/h_i]$ the coordinate ring of $U_i$ and let $B_i:=B
\otimes_A A_i$ be the coordinate ring of the pre-image of $U_i$ in
$Y$. Furthermore, let $f_i$ be the image of $f$ in $B_i$. Then $f_i \in
B_i \setminus A_i$ for some $i$.  If $\chr(K)=0$, then by the above we
find that $\Delta_1(f_i) \neq 0$, and this implies that $\Delta_1(f)
\neq 0$. If $\chr(K)>0$, then by the above we find a power $q$ of $p$
and an $i$ such that $\Delta_q(f_i) \neq 0$ and such that $f_j \in A_j
\cdot 
B_j^q$ for all $j$. This implies, first, that $\Delta_q(f) \neq 0$, and
second, that $h_j^{m_j} f \in A \cdot B^q$ for some $m_j \geq 0$. Since the
$U_j$ cover $X$, we have an expression $1=\sum_j a_j h_j^{m_j}$ where only finitely
many of the $a_j \in A$ are nonzero, and hence
\[ f=\sum_j a_j h_j^{m_j} f \in A \cdot B^q, \]
as desired. 
\end{proof}

\begin{remark}
One can define Hasse derivatives using $\Delta_n$, but we will not need them.
\end{remark}

\subsection{Subvarieties of torsors} \label{ss:subvartorsor}

Maintain the notation from \S \ref{ss:vecgp}. Let $Z \subset Y$ be a closed subscheme, with ideal $J \subset B$, and put $J_{\le 1}=J \cap B_{\le 1}$.

\begin{proposition} \label{prop:imm1}
Suppose that the map
\begin{displaymath}
\init_1 \colon J_{\le 1} \to A \otimes V
\end{displaymath}
is surjective. Then $Y \to X$ restricts to a closed immersion $Z \to X$.
\end{proposition}

\begin{proof}
We can verify the conclusion locally on $X$, and so we can assume that $Y$ is a trivial torsor. Thus $B=A \otimes \Sym(V)$. Our hypothesis means that for any $v \in V$, there is an element of $J$ of the form $1 \otimes v - a$ with $a \in A$. It follows that the map $A \to B/J$ is surjective, which is exactly what we want to show.
\end{proof}

The condition that $\init_1$ is surjective can be reformulated in a way that is sometimes more convenient:

\begin{proposition} \label{prop:imm2}
Suppose $K$ is algebraically closed, $X$ is of finite type over $K$ and $V$ is finite dimensional. Suppose moreover that for each $K$-point $y$ of $Y$ and each non-zero $r \in V^*$, there is some $f \in J_{\le 1}$ such that $(\partial_r f)(y) \ne 0$. Then $Z \to X$ is a closed immersion.
\end{proposition}

\begin{proof}
For a $K$-point $x$ of $X$, consider the map $\theta_x \colon J_{\le 1} \to V$ defined by first applying $\init_1$ and then evaluating the resulting element of $A \otimes V$ at $x$. If $r \in V^*$ then $\langle r, \theta_x(f) \rangle=(\partial_r f)(y)$ for any $K$-point $y$ of $Y$ above $x$ (note that $\partial_r f$ is constant on the fibers since $f \in J_{\le 1}$). Our assumption implies that $\theta_x$ is surjective; indeed, if the image of $\theta_x$ were contained in a proper subspace of $V$ then there would be some non-zero functional $r \in V^*$ vanishing on this subspace, and we would have $(\partial_r f)(x)=0$ for all $f \in J_{\le 1}$.

We have thus shown that the map
\begin{displaymath}
\init_1 \colon J_{\le 1} \to A \otimes V
\end{displaymath}
is surjective modulo $\fm$ for each maximal ideal $\fm$ of $A$. If $N$ is the cokernel of this map, then $N/\fm N=0$ for all $\fm$ which, by our assumptions, implies $N=0$. We thus see that $\init_1$ is surjective, and so we can apply Proposition~\ref{prop:imm1}.
\end{proof}

\subsection{Quotients}

Suppose now that we have an exact sequence of vector spaces
\begin{displaymath}
0 \to V_1 \to V_2 \to V_3 \to 0
\end{displaymath}
Note that $\bA(V_3)$ is a closed subgroup of $\bA(V_2)$, and that the
map $\bA(V_2) \to \bA(V_1)$ makes $\bA(V_2)$ into a trivial $\bA(V_3)$-torsor
over $\bA(V_1)$. 

\begin{proposition}
Let $\pi \colon Y \to X$ be an $\bA(V_2)$-torsor. Then $\pi$ naturally factors as $\pi_2 \circ \pi_1$, where $\pi_1 \colon Y \to Y'$ is an $\bA(V_3)$-torsor and $\pi_2 \colon Y' \to X$ is an $\bA(V_1)$-torsor.
\end{proposition}

\begin{proof}
First suppose that $X=\Spec(A)$ is affine, so that $Y=\Spec(B)$ is
also affine. Let $B' \subset B$ be the subalgebra consisting of functions that are invariant under $\bA(V_3)$, and put $Y'=\Spec(B')$. Clearly, $\pi$ factors as $Y \to Y' \to X$. Formation of $Y'$ is compatible with passing to open subsets of $X$, so the claims in the proposition can be checked locally. We can thus assume $Y=\bA(V_2) \times X$ is a trivial torsor. One finds $Y'=\bA(V_1) \times X$, and the result follows. We will not need the non-affine case, but it can be proved in a similar fashion.
\end{proof}

In the context of the above proposition, we call $Y'$ the \defn{quotient} of $Y$ by $\bA(V_3)$, and denote it $Y/\bA(V_3)$.

\subsection{Frobenius twists}

Suppose now that the characteristic $p$ of $K$ is non-zero. We investigate how $\bA(V)$-torsors interact with Frobenius twists.

\begin{proposition} \label{prop:torsor-frob}
Suppose that $Y=\Spec(B)$ is an $\bA(V)$-torsor over $X=\Spec(A)$, and let $q$ be a power of $p$. Then $Y^{(q)}$ is naturally an $\bA(V^{(q)})=\Spec(B^{(q)})$-torsor over $X$, and we have a natural identification $B^{(q)}=A \cdot B^q \subset B$.
\end{proposition}

\begin{proof}
By definition, $B^{(q)}=A \otimes_A B$, where $A \to A$ is the map $a
\mapsto a^q$. Let $F \colon B^{(q)} \to B,\ a \otimes b \mapsto ab^q$
be the Frobenius map for $B$ relative to $A$ (see \S\ref{ss:frob}).
The image of $F$ is $A \cdot B^q$. One checks locally that $F$ is injective, which proves that $B^{(q)} \cong A \cdot B^q$.

The comultiplication
\begin{displaymath}
\Delta \colon B \to B \otimes \Sym(V)
\end{displaymath}
is an $A$-algebra map. It clearly maps $AB^q$ into $AB^q \otimes \Sym(V)^q$. Identifying $\Sym(V)^q$ with $\Sym(V^{(q)})$, we thus have a map
\begin{displaymath}
\Delta^{(q)} \colon B^{(q)} \to B^{(q)} \otimes \Sym(V^{(q)}).
\end{displaymath}
Computing locally shows that this defines a torsor structure on $Y^{(q)}$.
\end{proof}

One useful application of Frobenius twists is that it allows us to reduce to the case where derivatives are non-zero, as the following proposition shows.

\begin{proposition} \label{prop:torsor-frob-2}
Maintain the notation from Proposition~\ref{prop:torsor-frob}. Suppose $f \in B$ has $\Delta_q(f) \ne 0$ and $\Delta_i(f)=0$ for $0<i<q$. Then $f \in B^{(q)}$ and $\Delta^{(q)}_1(f)=\Delta_q(f)$.
\end{proposition}

\begin{proof}
We have already seen (Proposition~\ref{prop:deriv}) that $f$ belongs to $A \cdot B^q=B^{(q)}$. The description of $\Delta^{(q)}$ given in the proof of Proposition~\ref{prop:torsor-frob} shows that $\Delta^{(q)}_i(f)=\Delta_{qi}(f)$. In particular, $\Delta_1^{(q)}(f)=\Delta_q(f)$, as claimed.
\end{proof}

Geometrically, the fact that $f$ belongs to $B^{(q)}$ means that $f$ descends to $Y^{(q)}$, i.e., $f$ is the pull-back of a function (which we still denote by $f$) under the map $Y \to Y^{(q)}$. Moreover, if we regard $f$ as a function on $Y^{(q)}$ then it has some non-vanishing directional derivative with respect to the action of $\bA(V^{(q)})$.

Here is how we will apply the above proposition. Suppose $Z$ is a closed subvariety of $Y$, and let $f$ be a function vanishing on $Z$. We can carry out a certain argument when $f$ has a non-vanishing derivative. Suppose $\Delta_1(f)=0$, and let $q$ be minimal such that $\Delta_q(f) \ne 0$. Then $f$ descends to $Y^{(q)}$, has non-zero derivative there, and vanishes on the scheme-theoretic image $Z'$ of $Z$ in $Y^{(q)}$. We can thus apply our argument to $Z'$.

\subsection{$\GL$-varieties}

Let $V$ be a polynomial representation. Then $\bA(V)$ is a group $\GL$-variety, and we can speak of $\bA(V)$-torsors in the category of $\GL$-schemes. To be precise, an $\bA(V)$-torsor over a (quasi-affine) $\GL$-scheme $X$ is a (quasi-affine) $\GL$-scheme $Y$ equipped with a map of $\GL$-schemes $Y \to X$ and an action of $\bA(V)$ such that the action map $\bA(V) \times Y \to Y$ is a map of $\GL$-schemes over $X$ and, ignoring the $\GL$-actions, makes $Y$ into a torsor over $X$. In other words, the local triviality condition is only enforced after forgetting the $\GL$-actions.

The results of this section can be applied directly to $\bA(V)$-torsors, since these are just ordinary torsors of the underlying schemes. This is why it is important that $V$ was allowed to be infinite dimensional in the preceding discussion. The constructions of this section also apply in the $\GL$-setting; for example, all natural maps will be $\GL$-equivariant.

We first give the main example of interest. Suppose we have a short exact sequence of polynomial representations
\begin{equation} \label{eq:ses}
0 \to Q \to P \to V \to 0.
\end{equation}
Then $\pi \colon \bA(P) \to \bA(Q)$ is an $\bA(V)$-torsor. This torsor is trivial if we forget the $\GL$-actions, since \eqref{eq:ses} splits in the category of vector spaces. However, if \eqref{eq:ses} fails to split in the category of representations, then $\pi$ is non-trivial as a $\GL$-equivariant torsor, meaning there is no isomorphism with the trivial torsor that preserves the $\GL$-actions. If $Z$ is a closed $\GL$-subvariety of $\bA(Q)$ then $\pi^{-1}(Z) \to Z$ is also an $\bA(V)$-torsor. The torsors we consider ultimately arise in this manner.

We will require the following version of Proposition~\ref{prop:imm2} in the $\GL$-setting:

\begin{proposition} \label{prop:imm3}
Consider the following situation:
\begin{itemize}
\item $K$ is algebraically closed.
\item $V$ is a finite length polynomial representation.
\item $X$ is an affine $\GL$-prevariety.
\item $\pi \colon Y \to X$ is an $\bA(V)$-torsor of affine $\GL$-schemes.
\item $Z$ is a closed $\GL$-subscheme of $Y$ with ideal $J$.
\end{itemize}
Suppose that for each $K$-point $y$ of $Y$ and each non-zero $r \in V^*$ there is some $f \in J_{\le 1}$ such that $(\partial_r f)(y) \ne 0$. Then $\pi$ restricts to a closed immersion $Z \to X$.
\end{proposition}

The proof is exactly the same as that of Proposition~\ref{prop:imm2}, except that we must use the following lemma:

\begin{lemma}
Let $A$ be a $\GL$-algebra that is finitely $\GL$-generated over $K$
and let $M$ be a finitely $\GL$-generated $A$-module such that
$M/\ker(x) M=0$ for all $K$-algebra homomorphisms $x:A \to K$. Then $M=0$.
\end{lemma}

\begin{proof}
It suffices to show that $M\{U\}=0$ for all finite dimensional subspaces
$U$ of $\bV$. Thus fix $U$. Since $M\{U\}$ is a finitely generated
$A\{U\}$-module, and $A\{U\}$ is a finitely generated $K$-algebra,
it suffices to show $M\{U\}/\fm M\{U\}=0$ for all maximal ideals $\fm$
of $A\{U\}$. Thus let $\fm$ be given. Since $A\{U\}$ is a $K$-algebra
of finite type and $K$ is algebraically closed, $\fm$ is the kernel of
a $K$-algebra homomorphism $x:A\{U\} \to K$. Let $\pi \colon \bV \to U$
be a linear surjection; this gives a surjective $K$-algebra homomorphism
$A \to A\{U\}$ also denoted $\pi$. By assumption, $M=\ker(x \circ \pi)
M$.
Applying $\pi$, we find $M\{U\}= \fm M\{U\}$, and so $M\{U\}/\fm M\{U\}=0$. The result follows.
\end{proof}

\section{The embedding theorem} \label{s:embed}

Suppose $f \colon Z \to X$ is a map of $\GL$-varieties. Essentially by definition, $f$ factors as a closed immersion $Z \to \bA(P) \times X$, for some polynomial representation $P$, followed by the projection map. Let $0=P_0 \subset \cdots \subset P_n=P$ be a composition series for $P$. Then the projection $\bA(P) \times X \to X$ factors as
\begin{displaymath}
\bA(P_n) \times X \to \bA(P_{n-1}) \times X \to \cdots \to \bA(P_0) \times X = X,
\end{displaymath}
and each map here is a torsor for some $\bA(P_i/P_{i-1})$. This shows that the study of general maps can be reduced to the study of closed subvarieties of torsors for irreducible representations. In \S \ref{s:embed}, we place ourselves in the latter situation, and prove the important embedding theorem (Theorem~\ref{thm:embed}).

\subsection{Statement of results}

We fix the following notation for the duration of this section:
\begin{itemize}
\item $R$ is an irreducible polynomial functor of degree $d>0$.
\item $X$ is an affine $\GL$-variety.
\item $\pi \colon Y \to X$ is an $\bA(R)$-torsor.
\item $Z$ is a closed $\GL$-subvariety of $Y$.
\end{itemize}
The purpose of this section is to study $Z$. We say $Z$ is \defn{cylindrical} if it has the form $\pi^{-1}(Z')$ for some closed $\GL$-subvariety $Z' \subset X$; equivalently, this means the ideal for $Z$ is extended from the ring $K[X]$. We will concentrate on the non-cylindrical case.

Shifting will be important in this section, so we make some comments
on it now. Let $U$ be a finite dimensional vector space.  We have
$\Sh_U(R)=R \oplus \Sh^{<d}_U(R)$, where the second term is a sum of
polynomial functors of degrees $<d$. The map $\Sh_U(Y) \to \Sh_U(X)$
is an $\bA(\Sh_U(R))$-torsor, and thus factors as
\begin{displaymath}
\Sh_U(Y) \to \Sh_U(Y)/\bA(R) \to \Sh_U(X),
\end{displaymath}
where the first map is an $\bA(R)$-torsor, and the second is an $\bA(\Sh^{<d}_U(R))$-torsor. Suppose $h$ is a function in $K[Y\{U\}]$. Then $h$ is naturally a $\GL$-invariant function on $\Sh_U(Y)$, and it is also $\bA(R)$-invariant, and thus descends to a function on the middle space above.

We are now able to state our main theorem:

\begin{theorem} \label{thm:embed}
Suppose $Z$ is non-cylindrical. There exists a finite dimensional vector space $U$ and a function $h \in K[Y\{U\}]$ that does not vanish identically on $Z\{U\}$ such that, putting $Y'=\Sh_U(Y)/\bA(R)$, the natural map $\Sh_U(Z)[1/h] \to Y'[1/h]$ is a closed $F$-immersion.
\end{theorem}

The proposition shows that if $Z$ embeds into an $\bA(R)$-torsor over $X$ then, after shifting, some non-empty open subvariety of $Z$ embeds into an $\bA(R')$-torsor over $X$, where $R'$ is a smaller polynomial representation than $R$. This provides an important tool for inductively studying $\GL$-varieties.

The rest of this section is devoted to the proof of this theorem. Our proof is modeled on the arguments in \cite[\S 2.9]{draisma}.

\subsection{The key result} \label{ss:embed-key}

The following is the key result we will use to prove the theorem.

\begin{proposition} \label{prop:embed-key}
Let $f \in K[Y]$ vanish on $Z$, let $r_0$ be a $K$-point of $\bA(R)$, and let $h=\partial_{r_0} f$. Let $U$ be a finite dimensional subspace of $\bV$ such that $f$ belongs to $K[Y\{U\}]$ and let $Y'=\Sh_U(Y)/\bA(R)$. Then $\pi$ induces a closed immersion $\Sh_U(Z)[1/h] \to Y'[1/h]$.
\end{proposition}

There are two points worth emphasizing about the statement of the
proposition. First, $h$ could vanish on all of $Z$, in which case the
conclusion of the proposition is uninteresting. Second, the
proposition is asserting that $\pi$ actually induces a closed
immersion, and not just a closed $F$-immersion.

To show that the map in question is a closed immersion, we can pass to the algebraic closure of $K$. We thus assume in what follows that $K$ is algebraically closed. Note this base change operation does not necessarily preserve reducedness, so in what follows $X$, $Y$, and $Z$ are affine $\GL$-prevarieties. We also use the notation $x \in X$ to mean that $x$ is a $K$-point of $X$ in the following.

Let $J \subset K[\Sh_U(Y)]$ be the ideal of $\Sh_U(Z)$, and let
$J_{\le 1}$ be the affine-linear elements of $J$ with respect to the
action of $R\{\bV\}^*$ (see \S\ref{ss:subvartorsor}). Both $J$ and
$J_{\le 1}$ are stable under the ind-group $\GL(U \oplus \bV):=\bigcup_n
\GL(U \oplus K^n)$, and $f$ belongs to $J$.

Suppose $\phi \colon U \to \bV$ is an injective linear map. Define
$\phi_t=1+t\phi$, regarded as a (lower-triangular) element of 
$\GL(U \oplus \bV)$. Working in the coordinate ring of $\Sh_U(Y)$, we have an expression
\begin{displaymath}
\phi_t \cdot f=f^{\phi}_0+t f^{\phi}_1 + \cdots + t^e f^{\phi}_e,
\end{displaymath}
for some $e \ge 0$, where $f^{\phi}_0=f$. Since $J$ is stable under $\GL(U
\oplus \bV)$, it follows that each $f_i^{\phi}$ belongs to $J$.  Recall
that $d$ is the degree of $R$. We will use stars to denote contragredient
actions, so for instance $\phi^*(r)$ is short-hand for $r \circ R(\phi)$.

\begin{lemma} \label{lem:embed-key-1}
For $y \in \Sh_U(Y)$ and $r \in R\{\bV\}^*$ and $u \in K$ we have
\begin{displaymath}
f^{\phi}_d(y+ur)=f_d^{\phi}(y)+u (\partial_{\phi^*(r)} f)(y).
\end{displaymath}
In particular, $f_d^{\phi}$ belongs to $J_{\le 1}$ and $\partial_r
f_d^{\phi}=\partial_{\phi^*(r)} f$.
\end{lemma}

\begin{proof}
We have $\phi_t(r)=t^d \phi^*(r)$, and so
\begin{align*}
(\phi_t \cdot f)(y+ur)
&= f(\phi_t^*(y)+ut^d \phi^*(r)) \\
&= f(\phi_t^*(y))+ut^d (\partial_{\phi^*(r)}f)(\phi_t^*(y)) + O(t^{2d}) \\
&= \sum_{i=0}^{d-1} t^i f_i^{\phi}(y) +
t^d(f^{\phi}_d(y)+u(\partial_{\phi^*(r)}f)(y)) + O(t^{d+1})
\end{align*}
and also
\begin{displaymath}
(\phi_t \cdot f)(y+ur) = \sum_{i=0}^e t^i f^{\phi}_i(y+ur).
\end{displaymath}
Comparing the coefficients of $t^d$, we obtain the stated identity. (Note: we used the assumption $d>0$ above when we converted $O(t^{2d})$ into $O(t^{d+1})$; this requires $2d \ge d+1$.)
\end{proof}

\begin{lemma} \label{lem:embed-key-2}
For any $y \in \Sh_U(Y)[1/h]$ there exists $r \in R\{\bV\}^*$ such that $(\partial_r f^{\phi}_d)(y) \ne 0$.
\end{lemma}

\begin{proof}
Since $\phi$ is injective, $\phi^* \colon R\{\bV\}^* \to R\{U\}^*$ is
surjective. We can thus find $r \in R\{\bV\}^*$ such that $\phi^*(r)=r_0$. We have
\begin{displaymath}
(\partial_r f^{\phi}_d)(y)=(\partial_{\phi^*(r)}f)(y)=h(y) \ne 0
\end{displaymath}
and so the result follows.
\end{proof}

\begin{lemma} \label{lem:embed-key-3}
For any $y \in \Sh_U(Y)[1/h]$ and any non-zero $r \in R\{\bV\}^*$ there exists $F \in J_{\le 1}$ such that $(\partial_r F)(y) \ne 0$.
\end{lemma}

\begin{proof}
Fix $y$ and $r$ and suppose the result is false; that is, for all $F \in J_{\le 1}$ we have $(\partial_r F)(y)=0$. Let $g \in \GL(\bV)$. Then $f^{g \phi}_d$ belongs to $J_{\le 1}$, and so we have
\begin{displaymath}
0 = (\partial_r f^{g\phi}_d)(y) = (\partial_{\phi^*(g^* r)} f)(y) = (\partial_{g^* r} f^{\phi}_d)(y).
\end{displaymath}
Let $N$ be the $\GL(\bV)$-submodule of $R\{\bV\}^*$ generated by $r$.
The above computation shows that $(\partial_{r'} f^{\phi}_d)(y)=0$ for
all $r' \in N$. Now, $R\{\bV\}^*$ need not be an irreducible
$\GL(\bV)$-module, but any non-zero submodule is dense in the sense
that it surjects on $R\{W\}^*$ for any finite dimensional subspace $W$ of $\bV$ (see Lemma~\ref{lem:embed-key-4} below). Thus $N$ is dense in this sense.

It now follows that $(\partial_{r'} f^{\phi}_d)(y)=0$ for all $r' \in
R\{\bV\}^*$, which contradicts the previous lemma. To see this, let
$r'$ be given. Let $W$ be a finite dimensional subspace of $\bV$ such
that $f^{\phi}_d$ belongs to the coordinate ring of $(\Sh_U{Y})\{W\}$.
There is some $r'' \in N$ that has the same restriction as $r'$ to
$R\{W\}^*$. It follows that $(\partial_{r'} f)(y)=(\partial_{r''}f)(y)=0$, as stated.
\end{proof}

\begin{lemma} \label{lem:embed-key-4}
Let $N$ be a non-zero $\GL(\bV)$-submodule of $R\{\bV\}^*$. Then for
any finite dimensional subspace $W$ of $\bV$, the natural map $N \to
R\{W\}^*$ is surjective.
\end{lemma}

\begin{proof}
Let $\lambda$ be a non-zero element of $N$, and let $W'$ be a finite
dimensional subspace of $\bV$ containing $W$ such that $\lambda$ is
non-zero on $R\{W'\}$. The image of $N \to R\{W'\}^*$ is then a
non-zero $\GL(W')$-submodule, and thus all of $R\{W'\}^*$ since
$R\{W'\}$ is a finite dimensional irreducible representation of
$\GL(W')$ (or zero). Since $R\{W'\}^* \to R\{W\}^*$ is surjective, the result follows.
\end{proof}

Proposition~\ref{prop:embed-key} now follows from Lemma~\ref{lem:embed-key-3} and Proposition~\ref{prop:imm3}.

\subsection{Proof of the main theorem}

We now prove Theorem~\ref{thm:embed}. The argument will be divided into several cases.

\textit{Case 1: $Z$ maps dominantly to $X$.} Let $f$ be a non-zero
function of minimal degree on $Y$ that vanishes on $Z$; note that some
such function exists since we have assumed $Z$ is non-cylindrical. Since
$Z$ maps dominantly to $X$, it follows that $f$ does not factor
through $\pi$; in other words, $f$ is not fixed by the action of
$\bA(R)$, and so $\Delta_i(f) \ne 0$ for some $i \ge 1$ (see
\S\ref{ss:vecgp}).

\textit{Case 1a: $\Delta_1(f) \ne 0$.} In this case, there is some $r_0
\in R^*$ (i.e., a $K$-point of $\bA(R)$), such that
$h=\partial_{r_0}(f)$ is non-zero; see \S\ref{ss:der}. Since $h$ has
smaller degree than $f$, it follows that $h$ does not vanish on $Z$. The
theorem follows in this case from Proposition~\ref{prop:embed-key}.

\textit{Case 1b: $\Delta_1(f) = 0$.} By Lemma~\ref{prop:deriv}, this can only happen when $K$ has positive characteristic $p$. Let $q$ be the power of $p$ in that lemma, so that $\Delta_q(f) \ne 0$. Let $F \colon Y \to Y^{(q)}$ be the Frobenius map relative to $X$. By Proposition~\ref{prop:torsor-frob}, $Y^{(q)}$ is an $\bA(R^{(q)})$-torsor over $X$. By Proposition~\ref{prop:torsor-frob-2}, $f$ descends to $Y^{(q)}$; in other words, there is a function $f'$ on $Y^{(q)}$ such that $f=F^*(f')$. Moreover, that proposition shows that $\Delta_1(f')$ is non-zero since its pullback under $F$ is $\Delta_q(f)$, which is non-zero. Let $r_0 \in (R^{(q)})^*$ be such that $h'=\partial_{r_0}(f')$ is non-zero, and put $h=F^*(h')$. Since $h$ has smaller degree than $f$, it does not vanish on $Z$. Since $F$ is a homeomorphism, $F(Z)$ is a closed subset of $Y^{(q)}$; give it the reduced subscheme structure, so that it is an irreducible $\GL$-variety. Clearly, $f'$ vanishes on $F(Z)$ and $h'$ does not.

Let $U$ be a finite dimensional subspace of $\bV$ such that $f$ belongs to the coordinate ring of $Y\{U\}$. Consider the diagram
\begin{displaymath}
\xymatrix{
\Sh_U(Z)[1/h] \ar[d] \ar[r] & (\Sh_U(Y)/\bA(R))[1/h] \ar[d] \\
\Sh_U(F(Z))[1/h'] \ar[r] & (\Sh_U(Y^{(q)})/\bA(R^{(q)}))[1/h'] }
\end{displaymath}
where the vertical maps are Frobenius. One readily verifies that the diagram commutes. Proposition~\ref{prop:embed-key} shows that the bottom map is a closed immersion. Since the vertical maps are $F$-equivalences, it follows that the top map is a closed $F$-immersion. This verifies the theorem in this case.

\textit{Case 2: the general case.} We now handle the case when $Z$ may not map dominantly to $X$. Let $X_0$ be the closure of $\pi(Z)$, given the reduced scheme structure; thus $X_0$ is an irreducible affine $\GL$-variety. Let $Y_0$ be the base change of $Y$ to $X_0$, which is an $\bA(R)$-torsor over $X_0$. Note that $Z$ is a non-cylindrical closed $\GL$-subvariety of $Y_0$. By Case~1, there are $U$ and $h_0$ such that the map
\begin{displaymath}
(\Sh_U{Z})[1/h_0] \to Y_0'[1/h_0]
\end{displaymath}
is a closed immersion (and the domain is non-empty); here $Y_0'=\Sh_U(Y_0)/\bA(R)$. Let $h$ be any extension of $h_0$ from $Y_0\{U\}$ to $Y\{U\}$; note that we can replace $h_0$ above with $h$. Since the natural map $Y_0'[1/h] \to Y'[1/h]$ is a closed immersion, the result follows.

\section{Elementary maps and varieties} \label{s:elem}

One of the main results of this paper is the decomposition theorem (Theorem~\ref{thm:decomp}), which decomposes an arbitrary map of $\GL$-varieties into ``elementary'' pieces. In \S \ref{s:elem}, we introduce various classes of elementary maps, and establish the basic properties about them that will be needed to prove the decomposition theorem.

\subsection{Definitions}

Let $f \colon Y \to X$ be a map of affine $\GL$-prevarieties.
\begin{enumerate}[label=(\Roman*)]
\item We say $f$ is \defn{purely elementary} if it factors as $f=f_n \circ \cdots \circ f_1$ where each $f_i$ has the structure of an $\bA(R_i)$-torsor, for some irreducible polynomial representation $R_i$ of positive degree. In this case, $f_0 \colon Y_0 \to X_0$ is an isomorphism.
\item We say $f$ is \defn{0-elementary} if $f_0 \colon Y_0 \to X_0$ is a surjective and smooth map of schemes, and $Y=Y_0 \times_{X_0} X$. In this case, $f$ itself is a smooth map of schemes; in particular, it is of finite type (in the usual sense).
\item We say $f$ is \defn{elementary} if it factors as $f_2 \circ f_1$, where $f_1$ is purely elementary and $f_2$ is 0-elementary. More explicitly, $f$ is elementary if and only if $f_0$ is smooth and surjective, and the map $Y \to Y_0 \times_{X_0} X$ is purely elementary.
\item We say $f$ is \defn{$F$-elementary} if it admits a factorization $f_2 \circ f_1$ where $f_1$ is an $F$-equivalence and $f_2$ is elementary.
\end{enumerate}
We say that a $\GL$-variety $X$ is \defn{elementary} or \defn{$F$-elementary} if the map $X \to \Spec(K)$ is so. Explicitly, $X$ is elementary if and only if $X_0$ is a smooth variety over $K$ and the map $X \to X_0$ factors into a sequence of torsors, and $X$ is $F$-elementary if and only if there exists an $F$-equivalence $X \to X'$ with $X'$ elementary.

\begin{remark}
If $\chr(K)=0$, and $f:Y \to X$ is purely elementary, then it follows from
the fact that polynomial functors are semisimple that $Y \cong X \times
\bA(P)$ for some polynomial functor $P$, where $f$ corresponds to the
projection. If $f$ is just elementary and $X=\Spec(K)$, then it follows that
$Y \cong Y_0 \times \bA(P)$, where $Y_0$ is smooth.  In \cite{polygeom},
where we developed the theory of $\GL$-varieties in characteristic zero,
an elementary $\GL$-variety is one of the form $B \times \bA(P)$,
where $P$ is a polynomial representation with $P_0=0$ and $B$ is an
irreducible affine variety over $K$. In the current paper, instead of requiring that $B$
is irreducible, we have the alternative condition that $B$ is smooth. This
condition is more convenient since it is preserved under base change.
\end{remark}

\subsection{An example}

We now give an example that illustrates some of the subtleties with elementary $\GL$-varieties when $K$ has positive characteristic $p$. Put
\begin{displaymath}
R=\Sym(\bV^{(p)} \oplus \Sym^{p^2}(\bV)).
\end{displaymath}
If $w$ belongs to $\bV^{(p)}$ or $\Sym^{p^2}(\bV)$, then we write $[w]$ for the corresponding element of $R$. Define $I$ to be the ideal of $R$ generated by the elements
\begin{displaymath}
[v^p]^p-[v^{p^2}]
\end{displaymath}
for $v \in \bV$. Let $S=R/I$ and let $Y=\Spec(S) \subset
\bA(\bV^{(p)}) \times \bA(\Sym^{p^2}(\bV))$. We show that: (i) $Y$ is $F$-elementary; (ii) ignoring the $\GL$-action, $Y$ is an affine space; and (iii) as a $\GL$-variety, $Y$ is not of the form $\bA(P)$ for any polynomial representation $P$.

The projection from $Y$
to the second factor $X=\bA(\Sym^{p^2}(\bV))$ is surjective. This can
be checked on $\ol{K}$-points: a $\ol{K}$-point $x \in X\{V\}$, where
$V$ is a finite dimensional $K$-vector space, is a symmetric tensor in
$(V^*_{\ol{K}})^{\otimes n}$. This gives rise to a $\ol{K}$-linear function
$\xi$ on $V^{(p)}_{\ol{K}}$ that sends $v^p$ to the pairing $\langle x,v^{p^2}
\rangle^{\frac{1}{p}}$. Then $(\xi,x)$ is a $K$-point of $Y\{V\}$.

Thus $K[X] \to K[Y]$ is injective. We claim that it is also
$F$-surjective. Indeed, since $[v^p]^p + I$ interesects $K[X]$ for every
$v \in \bV$, the $p$-th power of any element in $K[Y]$ is in the image
of $K[X] \to K[Y]$. We conclude that $Y \to X$ is an $F$-equivalence.
Since $X \to \Spec(K)$ is (purely) elementary, $Y$ is $F$-elementary.

Further, disregarding the $\GL$-action, $S$ is a graded polynomial ring
in the variables $e_1^p,e_2^p,\ldots \in \bV^{(p)}$ (of degree $p$)
and any complementary basis $x_1,x_2,\ldots$ in $\Sym^{p^2}(\bV)$
of the subspace $V^{(p^2)}$ spanned by $e_1^{p^2},e_2^{p^2},\ldots$
(of degree $p^2)$.

However, $S$ is not isomorphic, as a $\GL$-algebra, to any algebra
of the form $\Sym(P)$ for $P$ a polynomial representation.  Indeed,
if it were, then the degree-$d$ part $P_d$ of $P$ would be zero for
$d=0,\ldots,p-1,p+1,\ldots,p^2-1$, $P_p$ would have to be $\bV^{(p)}$,
and $\Sym^{p}(\bV^{(p)}) \oplus P_{p^2} \cong S_{p^2}$. But $S_{p^2}$
is the quotient of $\Sym^p(\bV^{(p)}) \oplus \Sym^{p^2}(\bV)$ by the
space spanned by the generators of $I$, and in this quotient, the image
of the first summand is not a direct summand, a contradiction.

\subsection{Permanence properties}

We now show that the classes of morphisms (I)--(IV) are closed under various natural operations.

\begin{proposition} \label{prop:elbc}
Consider a cartesian square of affine $\GL$-prevarieties
\begin{displaymath}
\xymatrix{
Y' \ar[r] \ar[d]_{f'} & Y \ar[d]^f \\
X' \ar[r] & X }
\end{displaymath}
If $f$ belongs to one of the classes (I)--(IV) then so does $f'$.
\end{proposition}

\begin{proof}
Case~(I) is clear, since the base change of an $\bA(R)$-torsor is still an $\bA(R)$-torsor. Now consider case~(II). Then $f'_0$ is the base change of $f_0$ along $X'_0 \to X_0$, and is thus surjective and smooth. We have $Y'_0=Y_0 \times_{X_0} X'_0$, and so
\begin{displaymath}
Y' = Y \times_X X' = (Y_0 \times_{X_0} X) \times_X X' = Y_0 \times_{X_0} X' = (Y_0 \times_{X_0} X_0') \times_{X_0'} X' = Y_0' \times_{X_0'} X'.
\end{displaymath}
Thus $f'$ is 0-elementary. Case~(III) follows immediately from the first two cases, and case~(IV) follows from case~(III) and the fact that $F$-equivalences are stable under base change (Proposition~\ref{prop:F-equiv-bc}).
\end{proof}

\begin{proposition}
Let $f \colon Y \to X$ be a morphism of affine $\GL$-prevarieties, and consider the induced morphism
\begin{displaymath}
f_{\red} \colon Y_{\red} \to X_{\red}
\end{displaymath}
on the reduced subschemes. If $f$ belongs to one of the classes (I)--(IV) then so does $f_{\red}$.
\end{proposition}

\begin{proof}
We first handle case~(I). Consider the cartesian square
\begin{displaymath}
\xymatrix{
Y' \ar[r] \ar[d]_{f'} & Y \ar[d]^f \\
X_{\red} \ar[r] & X }
\end{displaymath}
The map $f'$ is purely elementary since it is a base change of such a map. Since $X_{\red} \to X$ is a surjective closed immersion, so is $Y' \to Y$. Since $X_{\red}$ is reduced and $f'$ is a sequence of affine space torsors, it follows that $Y'$ is also reduced. Thus $Y'=Y_{\red}$, and necessarily $f'=f_{\red}$. We thus see that $f_{\red}$ is purely elementary.

We now consider case~(II). Consider the above cartesian square once again. Since $f$ is smooth, so is $f'$, and thus $Y'$ is reduced. Thus, as before, we see that $Y'=Y_{\red}$ and $f'=f_{\red}$. We thus see that $f_{\red}$ is 0-elementary, as it is the base change of such a map.

Case~(III) follows directly from cases~(I) and~(II), and case~(IV) follows from case~(III) and the fact that if $f$ is an $F$-equivalence then so is $f_{\red}$.
\end{proof}

\begin{proposition} \label{prop:elfrob}
Let $f \colon Y \to X$ be a morphism of affine $\GL$-prevarieties, let $q$ be a characteristic power, and consider the induced morphism
\begin{displaymath}
f^{(q)} \colon Y^{(q)} \to X^{(q)}
\end{displaymath}
on Frobenius twists (over $K$). If $f$ belongs to one of the classes (I)--(IV) then so does $f^{(q)}$.
\end{proposition}

\begin{proof}
For case (I), it suffices to observe that if $f:Y \to X$ is an
$\bA(R)$-torsor, where $R$ is an irreducible polynomial representation,
then $f^{(q)}: Y^{(q)} \to X^{(q)}$ is an $\bA(R^{(q)})$-torsor. The proof
is almost identical to the proof of Proposition~\ref{prop:torsor-frob},
but note that there $Y^{(q)}$ stands for the Frobenius twist relative to
$X$, while here $Y^{(q)}$ and $X^{(q)}$ stand for the Frobenius twists
relative to $\Spec(K)$. For case (II), we note that if $f_0$ is surjective
and smooth, then so is $f_0^{(q)}$, and if $Y=Y_0 \times_{X_0} X$, then
$Y^{(q)} = Y_0^{(q)} \times_{X_0^{(q)}} X^{(q)}$. Case (III) follows from
cases (I) and (II). Case (IV) follows from case (III) and the fact that
the Frobenius twist of an $F$-equivalence is again an $F$-equivalence
by Proposition~\ref{prop:frob-equiv}. 
\end{proof}

We now prove a few results that will allow us to change the order of 
compositions in the definitions above. 

\begin{lemma} \label{lem:elcomp-1}
Consider maps of affine $\GL$-prevarieties
\begin{displaymath}
f \colon Z \to Y, \qquad g \colon Y \to X
\end{displaymath}
where $f$ is 0-elementary and $g$ is purely elementary. Then $g \circ f$ is elementary.
\end{lemma}

\begin{proof}
Consider the composition
\begin{displaymath}
\xymatrix{
Z \ar[r]^-a & Z_0 \times_{X_0} X \ar[r]^-b & X }
\end{displaymath}
We have $g \circ f = b \circ a$. We show that $a$ is purely elementary and $b$ is 0-elementary, which will prove the lemma.

We begin with $a$. Since $f$ is 0-elementary, we have $Z=Z_0 \times_{Y_0} Y$, and since $g$ is purely elementary, we have $Y_0=X_0$. Thus $a$ is simply the base change of $g$ along $Z_0 \times_{X_0} X \to X$, and is therefore purely elementary.

To show that $b$ is 0-elementary, it suffices to show that $b_0$ is smooth and surjective. Now, $b_0$ is simply $g_0 \circ f_0$. Since $g_0$ is an isomorphism and $f_0$ is smooth and surjective, the result follows.
\end{proof}

\begin{lemma} \label{lem:elcomp-2}
Consider maps of affine $\GL$-prevarieties
\begin{displaymath}
f \colon Z \to Y, \qquad g \colon Y \to X
\end{displaymath}
where $f$ is elementary and $g$ is an $F$-equivalence. Assume that $X$ and $Y$ are reduced. Then $g \circ f$ is $F$-elementary.
\end{lemma}

\begin{proof}
Since $X$ and $Y$ are reduced and $g$ is an $F$-equivalence, some
Frobenius map $Y \to Y^{(q)}$ factors as $g' \circ g$
(Proposition~\ref{prop:frob-factor}). Consider the following diagram:
\begin{displaymath}
\xymatrix{
Z \ar[r]^a \ar[d]^f & X' \ar[r]^{a'} \ar[d]^b & Z^{(q)} \ar[d]^{f^{(q)}} \\
Y \ar[r]^g & X \ar[r]^{g'} & Y^{(q)} }
\end{displaymath}
Here $X'$ is defined to be the fiber product of $X$ and $Z^{(q)}$ over
$Y^{(q)}$, $a$ comes from the universal property of fiber products,
and both rows compose to the Frobenius morphisms. Since $f^{(q)}$
is elementary by Proposition~\ref{prop:elfrob}, it follows from
Proposition~\ref{prop:elbc} that $b$ is as well.

Since $g$ and $g' \circ g$ are $F$-equivalences, so is $g'$, and so is
$a'$ by base change (Proposition~\ref{prop:F-equiv-bc}). 
Since $a'$ and $a' \circ a$ are $F$-equivalences, so is $a$. We thus
see that $g \circ f = b \circ a$ is $F$-elementary.
\end{proof}

\begin{proposition} \label{prop:elcomp}
Consider maps of affine $\GL$-prevarieties
\begin{displaymath}
f \colon Z \to Y, \qquad g \colon Y \to X.
\end{displaymath}
If $f$ and $g$ are both of the same class (I)--(III) then so is $g \circ f$. The same is true for class (IV) provided that $X$ and $Y$ are reduced.
\end{proposition}

\begin{proof}
This is clear in case~(I). We now treat case~(II). It is clear that $(g \circ f)_0=g_0 \circ f_0$ is smooth and surjective. We have
\begin{displaymath}
Z = Z_0 \times_{Y_0} Y = Z_0 \times_{Y_0} (Y_0 \times_{X_0} X) = Z_0 \times_{X_0} X,
\end{displaymath}
and so $g \circ f$ is 0-elementary.

We now treat case~(III). Write $f=f_2 \circ f_1$ and $g=g_2 \circ g_1$ where $f_1$ and $g_1$ are purely elementary and $f_2$ and $g_2$ are 0-elementary. By Lemma~\ref{lem:elcomp-1}, we have $g_1 \circ f_2 = b \circ a$, where $b$ is 0-elementary and $a$ is purely elementary. Thus
\begin{displaymath}
g \circ f = g_2 \circ b \circ a \circ f_1.
\end{displaymath}
Since $g_2 \circ b$ is 0-elementary by case~(II) and $a \circ f_1$ is purely elementary by case~(I), it follows that $g \circ f$ is elementary.

We now treat case~(IV). Write $f=f_2 \circ f_1$ and $g=g_2 \circ g_1$
where $f_1$ and $g_1$ are $F$-equivalences and $f_2$ and $g_2$ are
elementary. Let $Z'$ be the target of $f_1$ and let $Y'$ be the target
of $g_1$. Since $g_2 \colon Y' \to X$ is elementary and $X$ is reduced, it follows that $Y'$ is reduced. Thus Lemma~\ref{lem:elcomp-2} shows that $g_1 \circ f_2$ is $F$-elementary, and so of the form $b \circ a$, where $a$ is an $F$-equivalence and $b$ is elementary. Thus
\begin{displaymath}
g \circ f = g_2 \circ b \circ a \circ f_1.
\end{displaymath}
Since $g_2 \circ b$ is elementary by case~(III) and $a \circ f_1$ is an $F$-equivalence, it follows that $g \circ f$ is $F$-elementary.
\end{proof}

\begin{proposition} \label{prop:elshift}
Let $f \colon Y \to X$ be a morphism of affine $\GL$-prevarieties, and consider the induced morphism
\begin{displaymath}
\Sh_n(f) \colon \Sh_n(Y) \to \Sh_n(X)
\end{displaymath}
on shifts. If $f$ belongs to one of the classes (II)--(IV) then so does $\Sh_n(f)$.
\end{proposition}

\begin{proof}
First suppose that $f$ is an $\bA(P)$-torsor for some polynomial representation $P$. Then $\Sh_n(f)$ is a torsor for $\bA(\Sh_n(P))$. Choosing a composition series for $\Sh_n(P)$ with simple quotients $R_1, \ldots, R_m$, we see that $\Sh_n(f)$ factors as $f_m \circ \cdots \circ f_1$, where $f_i$ is an $R_i$-torsor. If $R_i$ has positive degree then $f_i$ is purely elementary, while if $R_i$ has degree~0 (meaning it is the one-dimensional trivial representation) then $f_i$ is 0-elementary. We thus see that $\Sh_n(f)$ is elementary, as it is a composition of elementary maps. (This analysis shows why purely elementary maps are not closed under shifting.)

We now consider case~(II). Put $f'=\Sh_n(f)$, $X'=\Sh_n(X)$, and $Y'=\Sh_n(Y)$, and regard $\Sh_n$ as restricting to the subgroup $G(n)$. The map $f'_0$ is the map $f \colon Y\{K^n\} \to X\{K^n\}$ induced by $f$. Since this map is the base change of $f_0$ along $X\{K^n\} \to X_0$, it is smooth. We have
\begin{displaymath}
Y'_0 \times_{X'_0} X' = Y\{K^n\} \times_{X\{K^n\}} X.
\end{displaymath}
Since $Y\{K^n\} = Y_0 \times_{X_0} X\{K^n\}$, the above is $Y_0
\times_{X_0} X=Y'$. Thus $f'$ is 0-elementary.

Cases~(III) and~(IV) now follow easily.
\end{proof}

\begin{proposition} \label{prop:elopen}
Let $f \colon Y \to X$ be a map of affine $\GL$-prevarieties, let $U$ be a simple open of $Y$ such that $f_0(U_0)=X_0$, and let $g \colon U \to X$ be the restriction of $X$. If $f$ is of class (III) or (IV) then so is $g$.
\end{proposition}

\begin{proof}
First consider case~(III). Since $f_0$ is smooth so is $g_0$, and $g_0$ is surjective by assumption. Consider the square
\begin{displaymath}
\xymatrix{
U \ar[r] \ar[d] & Y \ar[d] \\
U_0 \times_{X_0} X \ar[r] & Y_0 \times_{X_0} X }
\end{displaymath}
Since $U$ is simple this square is cartesian. As the right map is
purely elementary, so is the left by Proposition~\ref{prop:elbc}. It follows that $g$ is elementary.

Now consider case~(IV). Write $f=f_2 \circ f_1$ where $f_1 \colon Y \to Y'$ is an $F$-equivalence and $f_2 \colon Y' \to X$ is elementary. Since $f_1$ is an $F$-equivalence, $U'=f_1(U)$ is a simple open of $Y'$. Of course, $f_1$ induces an $F$-equivalence $g_1 \colon U \to U'$. By case~(III), $f_2$ induces an elementary map $g_2 \colon U' \to X$. Thus $g=g_2 \circ g_1$ is $F$-elementary.
\end{proof}

\subsection{Generically elementary maps}

We say that a map $f \colon Y \to X$ of irreducible affine
$\GL$-varieties is \defn{generically ($F$-)elementary} if there are
non-empty simple open subsets $V \subset Y$ and $U \subset X$ such
that $f$ induces an ($F$-)elementary map $V \to U$. Note that $f$ is
then necessarily dominant.

\begin{proposition}
Consider a cartesian square of irreducible affine $\GL$-varieties
\begin{displaymath}
\xymatrix{
Y' \ar[r]^b \ar[d]_{f'} & Y \ar[d]^f \\
X' \ar[r]^a & X }
\end{displaymath}
If $f$ is generically ($F$-)elementary then so is $f'$.
\end{proposition}

\begin{proof}
Let $V \subset Y$ and $U \subset X$ be non-empty simple opens such that $f$ induces an ($F$-)elementary map $V \to U$. The square
\begin{displaymath}
\xymatrix{
b^{-1}(V) \ar[r] \ar[d]_{f'} & V \ar[d]^f \\
a^{-1}(U) \ar[r] & U }
\end{displaymath}
is cartesian, and so $f' \colon b^{-1}(V) \to a^{-1}(U)$ is ($F$-)elementary (Proposition~\ref{prop:elbc}). The result follows.
\end{proof}

\begin{lemma}
Let $f \colon Y \to X$ be a generically ($F$-)elementary map of irreducible $\GL$-varieties and let $V$ be a non-empty simple open subset of $Y$. Then $f$ induces a generically ($F$-)elementary map $V \to X$.
\end{lemma}

\begin{proof}
First suppose that $f$ is ($F$-)elementary. Since $f_0 \colon Y_0 \to
X_0$ is surjective, it follows that $g_0 \colon V_0 \to X_0$ is
dominant. Thus, by Chevalley's theorem, there is a non-empty open
affine subset $U_0$ of $X_0$ contained in $g_0(V_0)$. Let $U=U_0
\times_{X_0} X$, which is a simple open subset of $X$. The map
$f^{-1}(U) \to U$ is ($F$-)elementary (Proposition~\ref{prop:elbc}). Put $V'=V \cap f^{-1}(U)$. This is a simple open and the map $f_0 \colon V'_0 \to U_0$ is surjective. Thus $f$ induces an ($F$-)elementary map $V' \to U$ (Proposition~\ref{prop:elopen}), which shows that $f \colon V \to X$ is generically ($F$-)elementary.

We now treat the general case. Let $Y' \subset Y$ and $X' \subset X$ be simple opens such that $f$ induces an ($F$-)elementary map $Y' \to X'$. By the previous paragraph, $Y' \cap V \to X'$ is generically ($F$-)elementary. This clearly implies the same for $f \colon V \to X$.
\end{proof}

\begin{proposition} \label{prop:genelcomp}
Let $f \colon Z \to Y$ and $g \colon Y \to X$ be generically ($F$-)elementary maps of affine $\GL$-varieties. Then $g \circ f$ is generically ($F$-)elementary.
\end{proposition}

\begin{proof}
Let $W \subset Z$ and $V \subset Y$ be non-empty simple opens such
that $f$ induces an ($F$-)elementary map $W \to V$. By the previous
lemma, the restriction $g \colon V \to X$ is generically ($F$-)elementary. Let
$V' \subset V$ and $U \subset X$ be non-empty simple opens such that
$g$ induces an ($F$-)elementary map $V' \to U$. Let $W'=f^{-1}(V')
\cap W$, a nonempty simple open in $Z$.
Then $f$ induces an ($F$-)elementary map $W' \to V'$ (Proposition~\ref{prop:elbc}). The composition $W' \to V' \to U$ is ($F$-)elementary (Proposition~\ref{prop:elcomp}), which shows that $g \circ f$ is generically ($F$-)elementary.
\end{proof}

\begin{proposition} \label{prop:genelshift}
Let $f \colon Y \to X$ be a generically ($F$-)elementary map of irreducible affine $\GL$-varieties. Then $\Sh_n(f)$ is also generically ($F$-)elementary, for any $n \ge 0$.
\end{proposition}

\begin{proof}
Let $V \subset Y$ and $U \subset X$ be simple open subsets such that $f$
induces an ($F$-)elementary map $f|_V:V \to U$. Then $\Sh_n(V) \subset
\Sh_n(Y)$ and $\Sh_n(U) \subset \Sh_n(x)$ are simple open subsets and
$\Sh_n(f)$ induces the map $\Sh_n(f|_V):\Sh_n(V) \to \Sh_n(U)$
which is ($F$-)elementary by Proposition~\ref{prop:elshift}.
\end{proof}

Any finite dimensional variety can be regarded as a $\GL$-variety with trivial $\GL$ action. We will require the following result in this context. This result perhaps provides some intuition for why $F$-elementary is defined as it is.

\begin{proposition} \label{prop:gensmooth}
Let $f \colon Y \to X$ be a dominant map of irreducible affine varieties. Then $f$ is generically $F$-elementary.
\end{proposition}

\begin{proof}
By Noether normalization \stacks{0CBI}, we can factor $f$ as
\begin{displaymath}
Y \to X \times \bA^d \to X,
\end{displaymath}
where the first map is dominant and generically finite, and the second map is the projection. Since the second map is elementary, it suffices to show that the first map is generically $F$-elementary. Thus, relabeling and shrinking, we may assume that $f$ is a finite map.

Write $X=\Spec(A)$ and $Y=\Spec(B)$. Thus $A \to B$ is a module-finite map of integral domains. Let $L=\Frac(A)$ and $M=\Frac(B)$, so that $M/L$ is a finite field extension. Let $L'$ be the separable closure of $L$ in $M$. Thus $M/L'$ is purely inseparable, and so there is some characteristic power $q$ such that $M^q \subset L'$. Let $A'=AB^q$, and put $X'=\Spec(A')$; note that $\Frac(A')=L'$. The map $f$ factors as
\begin{displaymath}
\xymatrix{ Y \ar[r]^{f_1} & X' \ar[r]^{f_2} & X }
\end{displaymath}
Since $B^q \subset A' \subset B$, it follows that $f_1$ is an $F$-equivalence. The map $f_2$ is \'etale over the generic point of $X$ since $L'/K$ is a separable extension, and thus \'etale over some dense open subset of $X'$ \stacks{02GT}. This shows that $f_2$ is generically elementary. The result follows.
\end{proof}

\subsection{Locally elementary maps}

Let $f \colon Y \to X$ be a map of quasi-affine $\GL$-prevarieties. We say
that $f$ is \defn{locally elementary} if there exist $n$ and open affine
$G(n)$-subschemes $U \subset X$ and $V \subset Y$ such that $f$ maps
$V$ into $U$, $f:V \to U$ is $F$-elementary (regarded as a morphism of
$\GL$-prevarieties via the the isomorphism $\GL \to G(n)$), and such that
moreover $X=\GL \cdot U$ and $Y=\GL \cdot V$. We
say $X$ itself is locally elementary if the map $X \to \Spec(K)$ is.

\begin{example}
Let $X$ be the quasi-affine $\GL$-variety of $\infty \times
\infty$-matrices of rank precisely~$1$, equipped with the action $(g,A)
\mapsto gAg^{T}$. Then $X$ is locally elementary. Indeed, the open
subscheme 
$U :=X[1/x_{11}] \subset X$ where the top-left coordinate $x_{11}$ is
nonzero is an open, affine $G(1)$-subscheme and, as a $\GL$-variety via the isomorphism $\GL
\to G(1)$, isomorphic to $\Spec(K[x_{11},x_{11}^{-1}]) \times \bA(\bV
\oplus \bV)$ via the map that remembers only the first row and column
of the matrix. The projection $U \to \Spec(K[x_{11},x_{11}^{-1}])$ is
purely elementary, and the map $\Spec(K[x_{11},x_{11}^{-1}]) \to
\Spec(K)$ is $0$-elementary. Hence $U$ is elementary and 
$X=\GL \cdot U$ is locally elementary.
\end{example}

We will see in \S\ref{ssec:Decomp} that every morphism of
$\GL$-varieties is built up from locally elementary ones. A key
observation in that construction is the following.

\begin{proposition} \label{prop:loceltbasechange}
Locally elementary maps are preserved under base change.
\end{proposition}

\begin{proof}
Let $f \colon Y \to X$ be a locally elementary map of quasi-affine
$\GL$-prevarieties, and let $h:Z \to X$ be an arbitrary map of
quasi-affine $\GL$-prevarieties. Pick $n$ and affine, open
$G(n)$-subschemes $V \subset Y$ and $U \subset X$ such that $f$ maps $V$
into $U$, the restriction $f:V \to U$ is $F$-elementary (regarded
as a map of $\GL$-varieties via the isomorphism $\GL \to G(n)$), and
$X=\GL \cdot U$ and $Y=\GL \cdot V$. Then
in the following diagram all arrows marked $\subset$ are open
immersions, and the outer rectangle is cartesian:
\[ 
\xymatrix{
V \times_U h^{-1}(U) \ar[r] \ar[d]_{\subset} \ar@/_8ex/[ddd] & 
V \ar[d]^{\subset} \ar@/^8ex/[ddd] \\
Y \times_X Z \ar[r] \ar[d] & Y \ar[d]^f \\
Z \ar[r]_h & X\\
h^{-1}(U) \ar[u]^{\subset} \ar[r] & 
U. \ar[u]_{\subset} 
}
\]
The left-most arrow is $F$-elementary by
Proposition~\ref{prop:elbc}. Furthermore, we have 
\[ \GL \cdot (h^{-1}(U)) = h^{-1}(\GL \cdot U) = h^{-1}(X)=Z. \]
Similarly, one shows that $\GL \cdot (V \times_U h^{-1}(U))=Y \times_X
Z$. 
\end{proof}

\section{The main structure theorems} \label{s:shift}

In this section, we prove our main structural results about $\GL$-varieties: the shift theorem (Theorem~ \ref{thm:shift}), unirationality theorem (Theorem~\ref{thm:uni}), and decomposition theorem (Theorem~\ref{thm:decomp}). We give two applications of these results: a version of Chevalley's theorem (Theorem~\ref{thm:chevalley}), and a result about lifting points along maps of $\GL$-varieties (Theorem~\ref{thm:lift}). The latter is the key result needed in our application to strength.

\subsection{Finite type maps}

We say that a map $A \to B$ of $\GL$-algebras is of \defn{invariant finite presentation} if
\begin{displaymath}
B \cong A[x_1, \ldots, x_n]/(f_1, \ldots, f_m)
\end{displaymath}
where the elements $x_1, \ldots, x_n$ and $f_1, \ldots, f_m$ are $\GL$-invariant. Equivalently, the natural map $B_0 \otimes_{A_0} A \to B$ is an isomorphism. We say that a map $Y \to X$ of affine $\GL$-schemes is of \defn{invariant finite presentation} if the map on coordinate rings is; equivalently, the natural map $Y \to Y_0 \times_{X_0} X$ is an isomorphism. Note that a 0-elementary map has this property. Maps of invariant finite presentation are discussed in \cite[\S 3.3]{imgclosure} in characteristic~0, but much of that discussion applies more generally. In particular, we will require the following result proven there:

\begin{proposition} \label{prop:fintype}
Let $f \colon Y \to X$ be a map of affine $\GL$-varieties that, as a map of schemes, is of finite type. Then there exists $n \ge 0$ and a non-zero $\GL$-invariant function $h$ on $\Sh_n(X)$ such that $\Sh_n(Y)[1/h] \to \Sh_n(X)[1/h]$ is of invariant finite presentation.
\end{proposition}

\begin{proof}
See \cite[Proposition~3.6]{imgclosure}; the proof given there applies in the present setting.
\end{proof}

The following proposition is a key special case of the shift theorem, and is used in the proof of that theorem.

\begin{proposition} \label{prop:ft}
Let $f \colon Y \to X$ be a map of irreducible affine $\GL$-varieties that is dominant and of finite type. Then $\Sh_n(Y) \to \Sh_n(X)$ is generically $F$-elementary for $n \gg 0$.
\end{proposition}

\begin{proof}
First suppose that $f$ is of invariant finite presentation. The map $Y_0 \to X_0$ is generically $F$-elementary (Proposition~\ref{prop:gensmooth}). Let $V_0 \subset Y_0$ and $U_0 \subset X_0$ be non-empty open affine sets such that $V_0 \to U_0$ is $F$-elementary, and let $V=V_0 \times_{Y_0} Y$ and $U=U_0 \times_{X_0} X$. Since $f$ is of invariant finite presentation, the square
\begin{displaymath}
\xymatrix{
V \ar[r] \ar[d] & V_0 \ar[d] \\
U \ar[r] & U_0 }
\end{displaymath}
is cartesian. Since $V_0 \to U_0$ is $F$-elementary, so is $V \to U$ (Proposition~\ref{prop:elbc}). Thus $f$ is generically $F$-elementary, as required.

We now treat the general case. By Proposition~\ref{prop:fintype}, there are $n$ and $h$ such that $\Sh_n(Y)[1/h] \to \Sh_n(X)[1/h]$ is of invariant finite presentation. By the previous paragraph, this map is therefore generically $F$-elementary. This clearly implies that $\Sh_n(Y) \to \Sh_n(X)$ is generically $F$-elementary as well. Any further shift is still generically $F$-elementary by Proposition~\ref{prop:genelshift}.
\end{proof}

\subsection{The shift theorem}

The following is our first important structural result about $\GL$-varieties. It adheres to the general theme in representation stability that objects can be made simpler by shifting.

\begin{theorem} \label{thm:shift}
Let $f \colon Y \to X$ be a dominant map of irreducible affine $\GL$-varieties. Then $\Sh_n(f) \colon \Sh_n(Y) \to \Sh_n(X)$ is generically $F$-elementary for $n \gg 0$.
\end{theorem}

\begin{proof}
Consider the following statement, for a finite length polynomial representation $P$:
\begin{itemize}
\item[$S(P)$:] Given an irreducible affine $\GL$-variety $X$ and an
irreducible closed $\GL$-subvariety $Y$ of $X \times \bA(P)$ such that the projection $Y \to X$ is dominant, the map $\Sh_n(Y) \to \Sh_n(X)$ is generically $F$-elementary for $n \gg 0$.
\end{itemize}
If $f:Y \to X$ is any morphism of $\GL$-varieties, then $Y$ embeds into $X \times \bA(P)$ for
some $P$ in such a manner that $f$ is the restriction to $Y$ of the
projection to $X$, and so to prove the theorem it suffices to prove
$S(P)$ for all $P$. We proceed by induction on $P$ relative to the
partial order introduced in \S\ref{ss:poly}. When $P$ has degree~0 (i.e., it is a finite dimensional trivial representation of $\GL$) then $S(P)$ holds by Proposition~\ref{prop:ft}.

Now suppose that $P$ has positive degree $d$ and $S(P')$ holds for all $P'$ smaller than $P$. Let $X$ and $Y$ as in $S(P)$ be given. Choose a subrepresentation $Q$ of $P$ such that $R=P/Q$ is simple of degree $d$. Thus $X \times \bA(P)$ is an $\bA(R)$-torsor over $X \times \bA(Q)$. We consider two cases.

\textit{Case 1: $Y$ is cylindrical.} This means that there is a closed
$\GL$-subvariety $Y'$ of $X \times \bA(Q)$ such that $Y$ is the
inverse image of $Y'$. In particular, $Y \to Y'$ is an $\bA(R)$-torsor, and thus a purely elementary map. Since $S(Q)$ holds by induction, the map $\Sh_n(Y') \to \Sh_n(X)$ is generically $F$-elementary for $n \gg 0$. It follows that $\Sh_n(Y) \to \Sh_n(X)$ is as well (using Propositions~\ref{prop:genelcomp} and~\ref{prop:genelshift}).

\textit{Case 2: $Y$ is non-cylindrical.} We apply the embedding theorem (Theorem~\ref{thm:embed}). This states that the map
\begin{displaymath}
\Sh_n(Y)[1/h] \to (\Sh_n(X \times \bA(P))/\bA(R))[1/h]
\end{displaymath}
is a closed $F$-embedding for appropriate $h$ and $n$. We have
\begin{displaymath}
\Sh_n(X \times \bA(P))/\bA(R) = \Sh_n(X) \times \bA(Q')
\end{displaymath}
where $Q' \subset \Sh_n(P)$ is a subrepresentation with quotient $R$. Let $Y'$ be the image closure of $\Sh_n(Y)$ in $\Sh_n(X) \times \bA(Q')$. The map $\Sh_n(Y) \to Y'$ is an $F$-equivalence after inverting $h$, and thus is generically $F$-elementary. Since $S(Q')$ holds by induction, the map $\Sh_m(Y') \to \Sh_{n+m}(X)$ is generically $F$-elementary for $m \gg 0$. As in the previous case, we find that $\Sh_{n+m}(Y) \to \Sh_{n+m}(X)$ is generically $F$-elementary for $m \gg 0$.
\end{proof}

\begin{corollary} \label{cor:xelt}
Let $X$ be an irreducible affine $\GL$-variety. Then $\Sh_n(X)$ contains a non-empty open affine $F$-elementary $\GL$-variety for some $n \ge 0$.
\end{corollary}

\begin{corollary}
Let $X$ be an irreducible affine $\GL$-variety. Then there exists a polynomial $\delta_X \in \bQ[t]$ such that $\dim X\{K^n\} = \delta_X(n)$ for all $n \gg 0$.
\end{corollary}

\begin{proof}
There exists $m$ and a nonempty open affine $F$-elementary $\GL$-variety
$X' \subset \Sh_m(X)$. So $X' \to \Spec(K)$ factors as $f_2 \circ f_1$
where $f_1$ is an $F$-equivalence and $f_2$ is elementary. The source
and target of $f_1$ (evaluated at $K^n$) have the same dimension
and $f_2=f_4 \circ f_3$ where $f_3$ is purely elementary and $f_4$ is
$0$-elementary. Now $f_4$ is a morphism from some $N$-dimensional affine
variety to $\Spec(K)$ and $f_3$ is a finite composition of torsors with
fibres of the form $\bA(R_i)$ for polynomial representations $R_i$. Hence
\[ \dim(X\{K^{m+n}\})=\dim(X'\{K^n\})=N+\sum_i \dim(R_i(K^n)). \]
The corollary now follows from the fact that each $\dim(R_i\{K^n\})$ is a
polynomial in $n$.
\end{proof}

\begin{corollary} \label{cor:locelt}
Let $f: Y \to X$ be a dominant map of irreducible affine
$\GL$-varieties. Then there exist nonempty, open 
$\GL$-subvarieties $V \subset Y$ and $U \subset X$ such that $f$ induces
a locally $F$-elementary map $V \to U$.
\end{corollary}

\begin{proof}
By the shift theorem, there exist $n \geq 0$ and simple open subsets
$V' \subset \Sh_n(Y)$ and $U' \subset \Sh_n(X)$ such that $\Sh_n f$
induces an $F$-elementary map $V' \to U'$. Regarded as subschemes of $Y$
and $X$, respectively, $V'$ and $U'$ are $G(n)$-subschemes but typically
not $\GL$-subschemes. Define $V:=\GL \cdot V' \subset Y$ and
$U':=\GL \cdot U' \subset X$. Then $f$ maps $V$ into $U$
and the restriction $f:V \to U$ is locally $F$-elementary. 
\end{proof}

\subsection{The unirationality theorem}

Recall that a finite dimensional variety $X$ is called \defn{unirational} if there is a dominant map $\bA^n \to X$ for some $n$. In \cite{polygeom}, we showed that if $X$ is a $\GL$-variety in characteristic~0 then there is a dominant map $B \times \bA(P) \to X$ for some finite dimensional variety $B$. We called this the ``unirationality theorem,'' since it shows that $X$ is unirational up to finite dimensional error (the $B$ part). The following is our version of the unirationality theorem in positive characteristic:

\begin{theorem} \label{thm:uni}
Let $X$ be an irreducible affine $\GL$-variety. Then there is a dominant
morphism $Y \to X^{(q)}$ for some characteristic power $q$, where $Y$
is an elementary $\GL$-variety.
\end{theorem}

\begin{proof}
By Corollary~\ref{cor:xelt} there exists an $n$ such that $\Sh_n(X)$
contains a nonempty open affine $F$-elementary subvariety $Y$. Let $A$
be the coordinate ring of $X$. There is a natural surjective map $\Sh_n(X)
\to X$ of $\GL$-varieties corresponding to the injective map $A\{\bV\} \to
A\{K^n \oplus \bV\}$ coming from the inclusion $\bV \to K^n \oplus \bV,\
v \mapsto (0,v)$. The composition $Y \to \Sh_n(X) \to X$ is dominant,
and $Y \to \Spec(K)$ factors as $f_2 \circ f_1$ with $f_1:Y \to Y'$
an $F$-equivalence and $f_2:Y' \to \Spec(K)$ elementary. Let $B,B'$ be
the coordinate rings of $Y,Y'$, respectively, so that $B',A \subset B$.

By Proposition~\ref{prop:frob-factor}, $B^q \subset B'$ for some $q$,
and hence also the image of $A^q$ in $B$ is contained in $B'$. Now
$A^{(q)} \cong A^q$ via the Frobenius map (since $A$ is a domain), and
the inclusion $A^q \subset B'$ corresponds to a dominant morphism $Y'
\to X^{(q)}$ with $Y'$ elementary, as desired.
\end{proof}

See \cite[\S 5]{polygeom} for some corollaries of the unirationality theorem.

\subsection{The decomposition theorem} \label{ssec:Decomp}

Let $f \colon Y \to X$ be a morphism of quasi-affine $\GL$-varieties.
A \defn{locally elementary decomposition} (LED) of $f$ is a pair of
decompositions $Y=\bigsqcup_{j=1}^m Y_j$ and $X=\bigsqcup_{i=1}^n X_i$
where each $Y_j$ is an irreducible, locally closed $\GL$-subscheme in
$Y$ and similarly for the $X_i$; and such that for each $j$ there exists an
$i$ such that $f$ induces a locally $F$-elementary map $Y_j \to X_i$.
A locally elementary decomposition of a quasi-affine $\GL$-variety $Y$
itself is an LED of the map $Y \to \Spec(K)$.

The following theorem is our main structural result on $\GL$-varieties:

\begin{theorem} \label{thm:decomp}
Any map of quasi-affine $\GL$-varieties admits an LED.
\end{theorem}

\begin{proof}
Let $f:Y \to X$ be a map of quasi-affine $\GL$-varieties. In particular,
$Y$ is an open, dense $\GL$-subvariety of some affine $\GL$-variety
$Y'$. We proceed by noetherian induction on $Y'$. If $Y'$ is empty, then
we need only prove that $X$ admits a decomposition into irreducible,
locally closed $\GL$-subvarieties. This follows immediately from
noetherianity.

Otherwise, let $Y''$ be an irreducible component of $Y'$, and let $h_1$
be a nonzero function on $Y''$ that vanishes identically on $Y'' \setminus
Y$. Choose $n$ such that $h_1$ is $G(n)$-invariant.  Then $Y''[1/h_1]$
is an irreducible, open, affine $G(n)$-subvariety of $Y$.

Similarly, $X$ is a dense, open $\GL$-subvariety of some
affine $\GL$-variety $X'$. Let $X''$ be the closure in $X'$
of $f(Y''[1/h_1])$ and let $h_2$ be a nonzero function on $X''$ that
vanishes identically on $X'' \setminus X$. After increasing $n$ if
necessary, we may assume that $h_2$ is also $G(n)$-invariant, so that
$X''[1/h_2]$ is an irreducible, open affine $G(n)$-subvariety of $X$. 

Now $f$ induces a dominant morphism $Y''[1/(h_1 f^*(h_2))] \to X''[1/h_2]$
of irreducible affine $G(n)$-varieties.  By Corollary~\ref{cor:locelt}
there exist open $G(n)$-subvarieties $V \subset Y''[1/(h_1 f^*(h_2))]$ and
$U \subset X''[1/h_2]$ such that $f:V \to U$ is locally elementary
when regarded as a morphism of $\GL$-varieties via the isomorphism
$\GL \to G(n)$. Set $Y_0:=\GL \cdot V$ and $X_0:=\GL \cdot U$, so that
$X_0,Y_0$ are open $\GL$-subvarieties in $X,Y$, respectively, and $f:Y_0
\to X_0$ is locally elementary.

Set $Y_1:=(Y \setminus Y_0) \cap f^{-1}(X_0)$. Since $Y_1$ is contained in
a proper closed $\GL$-subvariety of $Y'$, the map $f:Y_1 \to X_0$
admits an LED, say $Y_1=\bigsqcup_j Y_{1j}$ and $X_0=\bigsqcup_i X_{0i}$. As
$X_0$ is irreducible, there is a unique $i_0$ for which $X_{0i_0}$ is open
in $X_0$. Now $Y_{0i_0} := f^{-1}(X_{0i_0}) \cap Y_0$ is open dense in $Y_0$,
and hence irreducible, 
and the restriction $f: Y_{0i_0} \to X_{0i_0}$ is locally elementary by
Proposition~\ref{prop:loceltbasechange}. Now combine  
\[ 
f: Y_{0i_0} \sqcup \bigsqcup_{j: f(Y_{1j})=X_{0i_0}} Y_{1j} \to X_{0i_0}
\]
with an arbitrary LED of the morphism 
\[ f: f^{-1}(X \setminus X_{0i_0}) \to X \setminus X_{0i_0}, \]
which, again, exists by the induction hypothesis. 
\end{proof}

\subsection{Chevalley's theorem}

Let $X$ be a $\GL$-variety. We say that a subset of $X$ is
\defn{$\GL$-constructible} if it is a finite union of locally closed
$\GL$-subvarieties. The following is our analog of Chevalley's theorem
for $\GL$-varieties:

\begin{theorem} \label{thm:chevalley}
Let $f \colon Y \to X$ be a map of quasi-affine $\GL$-varieties, and let $C$ be a $\GL$-constructible subset of $X$. Then $f(C)$ is a $\GL$-constructible subset of $Y$.
\end{theorem}

\begin{proof}
This follows easily from the decomposition theorem; see \cite[Theorem~7.13]{polygeom} for details.
\end{proof}

\subsection{Lifting points} \label{ssec:Lifting}

We now investigate the problem of lifting points through maps of $\GL$-varieties. The following is our main result:

\begin{theorem} \label{thm:lift}
Let $f \colon Y \to X$ be a map of quasi-affine $\GL$-varieties. There
exists $d \ge 1$ and a characteristic power $q$ with the following
property. If $x$ is an $L$-point of $X$, for some extension $L/K$,
that belongs to the image of $f$, then there exists an extension field $M/L$ of degree $\le d$ such that $x$ lifts to a $M^{1/q}$-point of $Y$.
\end{theorem}

\begin{proof}
By the decomposition theorem, it suffices to treat the case where $f$ is
locally elementary. Choose $n$ and open affine $G(n)$-subvarieties $V
\subset Y$ and $U \subset X$ such that $f$ induces an $F$-elementary
map $V \to U$ and $X=\GL \cdot U$ and $Y=\GL \cdot V$. Now the restriction
$f:V \to U$ factors as $f_3 \circ f_2 \circ f_1$, where $f_1$ is an
$F$-equivalence, $f_2$ is purely elementary, and $f_3$ is 0-elementary,
and hence the result for $L$-points $x \in U$ follows from the ensuing
lemma. The result for $L$-points of $X$ now follows from Proposition~\ref{prop:GLU}.
\end{proof}

\begin{lemma}
Let $f \colon Y \to X$ be a map of irreducible $\GL$-varieties.
\begin{enumerate}
\item If $f$ is 0-elementary then there exists $d$ with the following property: if $x$ is an $L$-point of $X$, for some extension $L/K$, then there exists an extension $M/L$ of degree at most $d$ and an $M$-point $y$ of $Y$ such that $f(y)=x$.
\item If $f$ is purely elementary then every $L$-point of $X$ lifts to a $L$-point of $Y$, for any extension $L/K$.
\item If $f$ is an $F$-equivalence then there exists a characteristic power $q$ such that every $L$-point of $X$ lifts to a $L^{1/q}$-point of $Y$, for any extension $L/K$.
\end{enumerate}
\end{lemma}

\begin{proof}
(a) Since $Y=Y_0 \times_{X_0} X$, to lift a point $x \in X(L)$ to $Y$ is equivalent to lifting its image $x_0 \in X_0(L)$ to $Y_0$. Since $Y_0 \to X_0$ is a surjective map of finite dimensional varieties, we see that $d$ exists.

(b) Ignoring the $\GL$-action, $Y$ is locally isomorphic to a product of $X$ with an affine space. Thus the claim is clear.

(c) By Proposition~\ref{prop:frob-factor} there is a characteristic power $q$ and a map $g \colon X \to Y^{(q)}$ such that $g \circ f$ is the Frobenius map $F \colon Y \to Y^{(q)}$. Let $x$ be a $K$-point of $X$. Then $g(x)$ is a $K$-point of $Y^{(q)}$. This lifts to a $K^{1/q}$-point $y$ of $Y$ satisfying $f(y)=x$.
\end{proof}

Recall that $K$ is called \defn{semi-perfect} if it has characteristic~0, or positive characteristic $p$ and $[K:K^p]$ is finite. In this case, we obtain a stronger result:

\begin{corollary} \label{cor:lift}
Suppose $K$ is semi-perfect and let $f \colon Y \to X$ be a map of $\GL$-varieties. Then there exists $e \ge 1$ with the following property: if $x$ is an $L$-point of $X$, for some algebraic extension $L/K$, that belongs to the image of $f$ then there exists an extension $M/K$ of degree at most $e$ such that $x$ lifts to an $M$-point of $Y$.
\end{corollary}

\begin{proof}
This follows from Theorem~\ref{thm:lift} and the subsequent lemmas.
\end{proof}

\begin{lemma}
Let $K$ be a semi-perfect field of positive characteristic $p$. Then for any algebraic extension $L/K$, we have $[L:L^p] \le [K:K^p]$. In particular, $L$ is semi-perfect.
\end{lemma}

\begin{proof}
Let $c=[K:K^p]$. First suppose that $L/K$ is finite. We have
\begin{displaymath}
[L:K^p]=[L:K][K:K^p] = [L:L^p][L^p:K^p].
\end{displaymath}
Since $L^p/K^p$ is isomorphic to $L/K$, the two extensions have the same degree, and so we find $[L:L^p]=c$ in this case.

We now treat the general case. Let $x_1, \ldots, x_{c+1}$ be elements of $L$, and let $M/K$ be the finite extension they generate. By the previous paragraph, $[M:M^p]=c$, and so there is a non-trivial linear dependence of the $x_i$'s with coefficients in $M^p \subset L^p$. Thus $[L:L^p] \le c$, as required. (Note: we can have $[L:L^p]<c$, e.g., if $L=\ol{K}$.)
\end{proof}

\begin{lemma}
Let $K$ be a semi-perfect field of positive characteristic $p$, let $d \ge 1$, and let $q$ be a characteristic power. Then there exists $e \ge 1$ with the following property: if $L/K$ is an algebraic extension and $M/L$ is an extension of degree at most $d$ then $[M^{1/q}:L] \le e$.
\end{lemma}

\begin{proof}
Let $c=[K:K^p]$ and write $q=p^k$. By the previous lemma, we have $[L:L^p] \le c$. Since $L^{p^i}/L^{p^{i+1}}$ is isomorphic to $L/L^p$, it also has degree $\le c$, and so
\begin{displaymath}
[L:L^q]=[L:L^p] [L^p:L^{p^2}]\cdots [L^{p^{k-1}}:L^{p^k}] \le c^k.
\end{displaymath}
We thus have
\begin{displaymath}
[M^{1/q}:L]=[M:L^q]=[M:L][L:L^q] \le d c^k,
\end{displaymath}
and so we can take $e=dc^k$.
\end{proof}

\begin{example}
In Corollary~\ref{cor:lift}, the condition that $L$ be semi-perfect is necessary, as the following example shows. Take $K=\bF_p$ and let $f \colon X \to X^{(p)}$ be the Frobenius map, with $X=\bA(\bV)$. Let $L=\bF_p(t_1, t_2, \ldots)$. Then the $L$-point $x=(t_1, t_2, \ldots)$ of $X^{(p)}$ has a unique lift to $X$, namely $y=(t_1^{1/p}, t_2^{1/p}, \ldots)$. The point $y$ is not defined over any finite extension of $L$.
\end{example}

\section{Application to strength} \label{s:strength}

Fix $d \ge 2$. Let $K[x_1, \ldots, x_n]_a$ denote the set of homogeneous forms of degree $a$. Recall that the \defn{strength} of $f \in K[x_1, \dots, x_n]_d$, denoted $\str(f)$, is the minimal $s$ such that
\begin{displaymath}
f = \sum_{i=1}^s g_i \cdot h_i,
\end{displaymath}
where $g_i$ and $h_i$ are homogeneous forms in $K[x_1, \ldots, x_n]$ of degrees $<d$. For an extension field $L/K$, we write $\str_L(f)$ for the strength of $f$ regarded as an element of $L[x_1, \ldots, x_n]_d$. We define the \defn{absolute strength} of $f$, denoted $\astr(f)$, to be $\str_{\ol{K}}(f)$.

The purpose of \S \ref{s:strength} is to prove Theorem~\ref{thm:astr2}, which we restate here:

\begin{theorem} \label{thm:strength}
Suppose $K$ is semi-perfect. Given any $s$ there is some $s'$ (depending only on $K$, $d$, and $s$) such that if $f \in K[x_1, \ldots, x_n]_d$ satisfies $\str(f)>s'$ then $\astr(f)>s$.
\end{theorem}

We now begin the proof. For $a \ge 0$, put
\begin{displaymath}
S_a = \bA(\Sym^a(\bV)).
\end{displaymath}
We identify $K$-points of $S_a\{K^n\}$ with elements of $K[x_1, \ldots, x_n]_a$. There are natural maps of $\GL$-varieties
\begin{displaymath}
S_a \times S_a \to S_a, \qquad S_a \times S_b \to S_{a+b}
\end{displaymath}
given by addition and multiplication. Put
\begin{displaymath}
X=S_d, \qquad Y=\coprod_{e=1}^{d-1} (S_e \times S_{d-e})
\end{displaymath}
and let $\theta^0 \colon Y \to X$ be the map induced by multiplication. Define
\begin{displaymath}
\theta_s \colon Y^s \to X, \qquad \theta_s(y_1, \ldots, y_s) = \theta^0(y_1)+\cdots+\theta^0(y_s).
\end{displaymath}
This is a map of $\GL$-varieties. Observe that $f \in K[x_1, \ldots, x_n]_d$ satisfies $\str_L(f) \le s$ if and only if $f$ is the image of an $L$-point of $Y^s\{K^n\}$ under $\theta_s$. In fact, we can verify this at infinite level too:

\begin{lemma}
Let $f \in K[x_1, \ldots, x_n]_d$, and regard $f$ as a $K$-point of $X$. Let $L/K$ be an extension. Then $\str_L(f) \le s$ if and only if $f$ is the image of an $L$-point of $Y^s$ under $\theta_s$.
\end{lemma}

\begin{proof}
Choose linear maps $i \colon K^n \to \bV$ and $\pi \colon \bV \to K^n$ satisfying $\pi \circ i = \id$. Consider the diagram
\begin{displaymath}
\xymatrix{
Y^s \ar[r] \ar@<2pt>[d]^{i^*} & X \ar@<2pt>[d]^{i^*} \\
Y^s\{K^n\} \ar@<2pt>[u]^{\pi^*} \ar[r] & X\{K^n\} \ar@<2pt>[u]^{\pi^*} }
\end{displaymath}
where the horizontal maps are induced by $\theta^s$. Both squares commute, and we have $i^* \circ \pi^*=\id$ for both vertical maps. Let $f'=\pi^*(f)$; this is how we regard $f$ as a $K$-point of $X$. If $f=\theta_s(y)$ for an $L$-point $y$ of $Y^s\{K^n\}$ then $f'=\theta_s(y')$ where $y'=\pi^*(y)$ is an $L$-point of $Y^s$. Similarly, if $f'=\theta_s(y')$ for an $L$-point $y'$ of $Y^s$ then $f=\theta_s(y)$ where $y=i^*(y')$ is an $L$-point of $Y^s\{K^n\}$.
\end{proof}

We require the following observation on strength:

\begin{lemma}
Let $f \in K[x_1, \ldots, x_n]_d$ and let $L/K$ be an extension of degree $e$. Then
\begin{displaymath}
\str_K(f) \le e \cdot \str_L(f).
\end{displaymath}
\end{lemma}

\begin{proof}
Suppose $\str_L(f) \le s$, and write
\begin{displaymath}
f = \sum_{i=1}^s g_i h_i,
\end{displaymath}
as usual. Let $\xi_1, \ldots, \xi_e$ be a $K$-basis for $L$
and write
$g_i = \sum_{j=1}^e g_{i,j} \xi_j$, 
where each $g_{i,j}$ is a homogeneous form in $K[x_1, \ldots, x_n]$ of the
same degree as $g_i$. 
Then we have 
\[ f=\sum_{i=1}^s \sum_{j=1}^e g_{ij} \cdot (\xi_j h_i). \]
This can be read as a system of linear equations over $K$ of which the
tuple $(\xi_j h_i)_{i,j}$ is a solution over $L$. But then the system also has
a solution $(\tilde{h}_{ij})_{i,j}$ over $K$, and we find $\str_K(f)
\leq es$. 
\end{proof}

We are now ready to prove the theorem.

\begin{proof}[Proof of Theorem~\ref{thm:strength}]
By Corollary~\ref{cor:lift} there is some $e \ge 1$ such that any $K$-point of $X$ in the image of $\theta_s$ lifts to an $M$-point of $Y^s$, for some extension field $M/K$ of degree $\le e$. Suppose now that $f \in K[x_1, \ldots, x_n]_d$ satisfies $\astr(f) \le s$. Then $f$ is the image of a $\ol{K}$-point under $\theta^s$. Thus it is the image of an $M$-point, for some $M/K$ with $[M:K] \le e$. We thus find
\begin{displaymath}
\str_K(f) \le e \cdot \str_M(f) \le e s.
\end{displaymath}
We can therefore take $s'=e s$: indeed the contrapositive of the above discussion shows that $\str(f)>s'$ implies $\astr(f)>s$.
\end{proof}


\begin{thebibliography}{BDES1}

\bibitem[AH]{AH} Tigran Ananyan, Melvin Hochster. Small subalgebras of polynomial rings and Stillman’s conjecture. \textit{J.~Amer.\ Math.\ Soc.} \textbf{33} (2020), pp.~291--309. \DOI{10.1090/jams/932} \arxiv{1610.09268v3}

\bibitem[BDD]{BDD} Arthur Bik, Alessandro Danelon, Jan Draisma. Topological Noetherianity of polynomial functors II: base rings with Noetherian spectrum. \textit{Math.\ Ann.} \textbf{385} (2023), pp.~1879--1921. \DOI{10.1007/s00208-022-02386-9} \arxiv{2011.12739}

\bibitem[BDLZ]{BDLZ} Arthur Bik, Jan Draisma, Amichai Lampert, Tamar Ziegler. Strength and partition rank under limits and field extensions. In preparation.

\bibitem[BDES1]{polygeom} Arthur Bik, Jan Draisma, Rob H.~Eggermont, Andrew Snowden. The geometry of polynomial representations. \textit{Int.\ Math.\ Res.\ Not.\ IMRN} (2022). \DOI{10.1093/imrn/rnac220} \arxiv{2105.12621}

\bibitem[BDES2]{imgclosure} Arthur Bik, Jan Draisma, Rob H.~Eggermont, Andrew Snowden. Uniformity for limits of tensors. \arxiv{2305.19866}

\bibitem[BDES3]{unirat} Arthur Bik, Jan Draisma, Rob H.~Eggermont, Andrew Snowden. Improved unirationalty for $\GL$-varieties. In preparation.

\bibitem[BDS1]{BDS1} Arthur Bik, Jan Draisma, Andrew Snowden. Two improvements in Brauer's theorem on forms. \arxiv{2401.02067}

\bibitem[BDS2]{BDS2} Arthur Bik, Jan Draisma, Andrew Snowden. Universality of high-strength tensors over non-closed fields. In preparation.

\bibitem[CM]{CM} Alex Cohen, Guy Moshkovitz. Partition and analytic rank are equivalent over large fields. \textit{Duke Math.\ J.} \textbf{172}, no.~12 (2023), pp.~2433--2470. \DOI{10.1215/00127094-2022-0086} \arxiv{2102.10509}

\bibitem[Dr]{draisma} Jan Draisma. Topological Noetherianity of polynomial functors. \textit{J.\ Amer.\ Math.\ Soc.}\ \textbf{32}(3) (2019), pp.\ 691--707. \DOI{10.1090/jams/923} \arxiv{1705.01419}

\bibitem[DLL]{DLL} Jan Draisma, Micha\l{} Laso\'n, Anton Leykin. Stillman's conjecture via generic initial ideals. \textit{Commun.\ Algebra} \textbf{47} (2019), no.~6, pp.~2384--2395. \DOI{10.1080/00927872.2019.1574806} \arxiv{1802.10139}

\bibitem[ESS]{ESS1} Daniel Erman, Steven V Sam, Andrew Snowden. Big polynomial rings and Stillman's conjecture. \textit{Invent.\ Math.} \textbf{218} (2019), no.~2, pp.~413--439. \DOI{10.1007/s00222-019-00889-y} \arxiv{1801.09852}

\bibitem[FS]{FS} Eric M.~Friedlander, Andrei Suslin. Cohomology of finite group schemes over a field. \textit{Invent.~Math.} \textbf{127}, no.~2 (1997), pp.~209--270. \DOI{10.1007/s002220050119}

\bibitem[Gan1]{Ganapathy1} Kathik Ganapathy. $\GL$-algebras in positive characteristic II: The polynomial ring. In preparation.

\bibitem[Gan2]{Ganapathy2} Karthik Ganapathy. $\GL$-algebras in positive characteristic I: The exterior algebra. \textit{Selecta Math.\ (N.S.)}, to appear. \arxiv{2203.03693}

\bibitem[GT]{GT} Ben Green, Terence Tao. The distribution of polynomials over finite fields, with applications to the Gowers norms. \textit{Contrib.\ Discrete Math.}\ \textbf{4}, no.~2 (2009), pp.~1--36. \DOI{https://doi.org/10.11575/cdm.v4i2.62086} \arxiv{0711.3191}

\bibitem[KaZ]{KaZ} David Kazhdan, Tamar Ziegler. Applications of algebraic combinatorics to algebraic geometry. \textit{Indag.\ Math.\ (N.S.)} \textbf{32} (2021), no.~6, pp.~1412--1428. \DOI{10.1016/j.indag.2021.09.002} \arxiv{2005.12542}

\bibitem[KLP]{KLP} David Kazhdan, Amichai Lampert, Alexander Polishchuk. Schmidt rank and singularities. \textit{Ukrainian Math. J.}, to appear. \arxiv{2104.10198}

\bibitem[KP]{KP} David Kazhdan, Alexander Polishchuk.  Schmidt rank of quartics over perfect fields. \textit{Isr. J. Math.} \textbf{255}, no.~2 (2023), pp.~851--869. \DOI{10.1007/s11856-022-2457-5} \arxiv{2110.10244}

\bibitem[LZ1]{LZ1} Amichai Lampert, Tamar Ziegler. Relative rank and regularization. \textit{Forum Math.\ Sigma}, to appear. \arxiv{2106.03933}

\bibitem[LZ2]{LZ2} Amichai Lampert, Tamar Ziegler. On rank in algebraic closure. \textit{Selecta Math.\ (N.S.)}, to appear. \arxiv{2205.05329}

\bibitem[M]{M} Luka Mili{\'c}evi{\'c}. Polynomial bound for partition rank in terms of analytic rank. \textit{Geom.\ Funct.\ Anal.} \textbf{29}, no.~5 (2019), pp.~1503--1530. \DOI{10.1007/s00039-019-00505-4} \arxiv{1902.09830}

\bibitem[MZ]{MZ} Guy Moshkovitz, Daniel G.~Zhu. Quasi-linear relation between partition and analytic rank.  \arxiv{2211.05780} 

\bibitem[Sch]{Schmidt} Wolfgang M. Schmidt. The density of integer points on homogeneous varieties. \textit{Acta Math.} \textbf{154} (1985), no.~3--4, pp.~243--296. \DOI{10.1007/BF02392473}

\bibitem[Sno]{Snowden} Andrew Snowden. Stable representation theory: beyond the classical groups. \arxiv{2109.11702}

\bibitem[Stacks]{stacks} The Stacks Project Authors. Stacks project (2023). {\tiny\url{https://stacks.math.columbia.edu}}

\bibitem[SS1]{symc1} Steven V Sam, Andrew Snowden. GL-equivariant modules over polynomial rings in infinitely many variables. \textit{Trans.\ Amer.\ Math.\ Soc.} \textbf{368} (2016), pp.~1097--1158. \DOI{10.1090/tran/6355} \arxiv{1206.2233}

\bibitem[SS2]{symu1} Steven V Sam, Andrew Snowden. GL-equivariant modules over polynomial rings in infinitely many variables. II. \textit{Forum Math.\ Sigma} \textbf{7} (2019), e5, 71 pp. \DOI{10.1017/fms.2018.27} \arxiv{1703.04516}

\bibitem[Tou]{Touze} Antoine Touz{\'e}. Foncteurs strictement polynomiaux et applications. Habilitation thesis, Universit\'e Paris 13, Foncteurs strictement polynomiaux et applications,
{\tiny\url{https://math.univ-lille1.fr/~touze/NotesRecherche/Habilitation_A_Touze.pdf}}, 2014.

\end{thebibliography}
\end{document}